\documentclass[14pt]{extarticle} 
\usepackage[T1]{fontenc}
\usepackage[utf8x]{inputenc}
\usepackage{babel}
\usepackage[unicode=true,pdfusetitle,
bookmarks=true,bookmarksnumbered=false,bookmarksopen=false,
breaklinks=false,pdfborder={0 0 1},backref=false,colorlinks=true, allcolors=blue]
{hyperref}
\usepackage{geometry}
\usepackage{verbatim}\usepackage{float}
\usepackage{enumitem}
\input{xy}
\xyoption{all}
\usepackage[all]{xy} 
\usepackage{amsmath}
\usepackage{graphicx}
\usepackage{mathtools}
\usepackage[colorinlistoftodos]{todonotes}
\usepackage{amssymb}
\usepackage{amsthm}
\theoremstyle{plain}
\newtheorem{theorem}{Theorem}

\newtheorem{thm}{\textsf{\textbf{Theorem}}}[section]
\newtheorem{lem}[thm]{\textsf{\textbf{Lemma}}}
\newtheorem{cor}[thm]{\textsf{\textbf{Corollary}}}
\newtheorem{prop}[thm]{\emph{Proposition}}

\newtheorem*{thm*}{\textsf{\textbf{Theorem}}}
\newtheorem*{prop*}{\emph{Proposition}}
\newtheorem*{claim*}{\emph{Claim}}
\theoremstyle{definition}
\newtheorem{dfn}[thm]{\textbf{\textsf{Definition}}}

\newtheorem{rem}[thm]{{\textsf{Remark}}}

\newcommand{\Aaa }{\mathcal A}
\newcommand{\Raa }{\mathcal R}

\newcommand{\Uee }{\mathcal U}

\newcommand{\nat }{\mathbb N}
\newcommand{\real }{\mathbb R}

\DeclareMathOperator{\inter}{Int}

\DeclareMathOperator{\cl}{Cl}
\newcommand{\fr}{\partial}

\DeclareMathOperator{\diam}{diam}

\title{Strange attractors for the family of orientation preserving Lozi maps}
\author{Przemys{\l}aw Kucharski}
\date{}

\newcommand{\norma}[1]{\parallel#1\parallel}
\newcommand{\modul}[1]{|#1|}
\DeclareMathOperator{\sgnt}{sgn}
\newcommand{\sgn}[1]{\sgnt(#1)}
\newcommand{\clu}[1]{\bar{#1}}

\newcommand{\stmani}[1]{W^{s}_{#1}} 
\newcommand{\ustmani}[1]{W^{u}_{#1}} 
\newcommand{\stlocmani}[1]{W^{s}_{#1,loc}}
\newcommand{\ustlocmani}[1]{W^{u}_{#1,loc}}
\newcommand{\seg}[1]{\overline{#1}}
\newcommand{\omegalim}[1]{\omega(#1)}

\newcommand{\plane }{\mathbb R^{2}}

\newcommand{\dH}[2]{d_{H}(#1,#2)}
\newcommand{\ball}[2]{B_{#2}(#1)}

\newcommand{\exrotation}[1]{\rho(#1)}
\newcommand{\lebmes}[1]{l(#1)}

\providecommand{\affiliations}[1]
{
	\small	
	\textbf{\textit{Author's affiliation}---} #1
}
\begin{document}
		\maketitle
		\date{}
		\affiliations{Faculty of Mathematics and Computer Science,
			Jagiellonian University in Krak\'{o}w
			ul. {\L}ojasiewicza 6, 30-348 Krak\'{o}w, Poland}
		\begin{abstract}
			We extend the result of Micha{\l} Misiurewicz assuring the existence of strange attractors for the parametrized family $\{f_{(a,b)}\}$ of orientation reversing Lozi maps to the orientation preserving case. That is, we rigorously determine an open subset of the parameter space for which an attractor $\Aaa_{(a,b)}$ of $f_{(a,b)}$ always exists and exhibits chaotic properties. Moreover, we prove that the attractor is maximal in some open parameter region, and arises as the closure of the unstable manifold of a fixed point, on which $f_{(a,b)}|_{\Aaa_{(a,b)}}$ is mixing. We also show that $\Aaa_{(a,b)}$ vary continuously with parameter $(a,b)$ in the Hausdorff metric.
		\end{abstract}
\textbf{ In the theory of dynamical systems, attractor is a set of states, which system is expected to approach with the progress of time, for a considerable amount of initial conditions. Moreover, states neighbouring an attractor tend to stay close to it as system evolves. Therefore, attractors are of particular interest to researchers, since they can be thought of as a set of states approximating evolution of a system and hence reducing analysis to study of certain subsystems. Especially important to the field seems to be strange attractors, which are subsystems additionally exhibiting some kind of chaotic behaviour. Strange attractors have been observed in physical electronic chaotic circuits, models for atmospheric convection or chemical reactions. In our paper we study strange attractors arising from the Lozi maps, a family of piecewise affine homemorphisms of the plane. Study of the Lozi family helps researchers develop tools for analysis of less amenable systems. An example would be the H\'{e}non family, which plays important role in the theory of continuous time dynamical systems modelling atmospheric convection (Lorenz systems). In this paper we show the existence of strange attractors for a wide variety of parameters, for which the Lozi family is orientation preserving, generalizing results of Michał Misiurewicz. A significant difference between our and other similar results in the field is that we additionally show that the attractors are maximal in a certain way, which effectively reduces study of the whole system to a subsystem restricted to an attractor.}
\section*{Introduction}
The family of Lozi maps is a parametrized family of piecewise linear planar homeomorphisms given by $f_{(a,b)}(x,y)=(1+y-a|x|,bx)$ for $a,b\in \mathbb{R}$. It has been introduced in 1978 by Ren\'e Lozi as a simplification of H\'{e}non family, potentially sharing some of its properties but being more approachable. Originally, the most interesting feature of the H\'{e}non and Lozi families was the fact that they approximated the Poincar\'{e} map of the Lorenz system \cite{henon1,henon2}, assuming a transversal is properly chosen.  Numerical experiments showed the existence of attractors for certain parameters \cite{henon2, lozi-first-article}. As discrete dynamical systems are significantly more amenable then continuous, H\'{e}non and Lozi families posed a new line of research, ultimately directed towards better understanding of attractors of continuous dynamical systems like Lorenz, R{\"o}ssler \cite{ROSSLER1976397} Chua \cite{chua}. Confining our attention to the Lozi family, in 1980  Misiurewicz \cite{strange-attractor-mis} (see Theorem \ref{misiurewicz-main}) determined a subset of the parameter space, for which $f_{(a,b)}$ is orientation reversing ($b>0$), where for every $(a,b)$ exists a strange attractor for $f_{(a,b)}$, that is an attractor on which $f_{(a,b)}$ is transitive. Later, Misiurewicz and Sonja \v{S}timac improved this result, extending the set of parameters \cite{Lozi-likemaps}. Among other directions of research we can point to the development of kneading theory \cite{Towardsakneading1},\cite{Towardsakneadingtheory2}, relationship to symbolic dynamics \cite{mm}  determining entropy and its continuity or lack of it \cite{Towardsakneadingtheory2}, \cite{discon-lozi-map},\cite{mono-lozi-map}, robustness of chaos and continuity of attractors \cite{glendinning2019robust} or characterisation as inverse limits of trees \cite{densly-branching-trees}. 

It was expected that there exists an orientation preserving counterpart of Misiurewicz's results. Some properties of the orientation preserving Lozi family, that one would be concerned in the proof of existence of strange attractors, has been previously established. Although usually context of said studies was not strange attractors of the Lozi family \textit{per se}. Specifically, hyperbolicity and the existence of invariant cones families \cite{rychlik:invariant-measures,simpson:robust-chaos, simpson:devaney-chaos}, presence of a trapping region, containing the unstable manifold of a fixed point \cite{rychlik:invariant-measures,simpson:robust-chaos,Glendinning2017,cao-liu-98} was considered. The existence of chaotic attractors for the orientation preserving 2-dimensional border-collision normal forms was shown in \cite{simpson:robust-chaos}. It is important to point to some differences between our paper and \cite{simpson:robust-chaos}. In \cite{simpson:robust-chaos} the authors' definition of a chaotic attractor is that of an invariant, compact subset, possessing a trapping region and on which system exhibits Devaney chaos (in particular it is transitive). Proofs in \cite{simpson:robust-chaos} do not show existence of a strange attractor in the sense of definition introduced in Section \ref{sec:attractors}, which is also considered for example in \cite{Lozi-likemaps}. In other words, they do not show that the attractor is maximal. In particular, we do not know if there are other invariant sets in the trapping region of \cite{simpson:robust-chaos}. For a detailed discussion of 2-dimensional border-collision normal form we refer the reader to \cite{simpson:robust-chaos}. While, we were unable to prove that the attractor present in our system is a maximal one for a large, constructively defined parameter region, we do prove that for a certain open set of parameters. The obtained set is in a way similar to counterpart results of the orientation preserving H{\'e}non maps \cite{benedicks-viana:solution-basin-henon,cao-mao:non-wandering-set-of-some-Henon-maps}. One should mention also a paper \cite{cao-liu-lozi} by Yongluo Cao and Zengrong Liu, in which they show that the realm of attraction is the whole trapping region, that is, the attractor is maximal. Unfortunately, it seems that their Proposition 6. contains a gap in reasoning, rendering the argument simply insufficient. Interestingly, this insufficiency is in line with a numerical counterexample provided by Glendinning and Simpson \cite{simpson:robust-chaos}.

Summarizing, Misiurewicz's results about strange attractors were never fully generalized to the orientation preserving case and so the purpose of this paper is to fill this gap. Moreover, as methods and even definitions vary significantly in cited papers, we believe that self-contained proof that is specific to the Lozi family, is of interest to the reader and so we treat all steps in our investigations.
 
 We now outline the Misiurewicz's proof of the existence of strange attractors. Lozi map is not hyperbolic in the usual sense, as it is not everywhere differentiable, but one can show that almost every point has a local unstable and stable manifolds and they can be extended to global manifolds that are broken lines. It allowed Misiurewicz to find an attracting neighbourhood of the unstable manifold of fixed point $X$. Furthermore, the investigation of homoclinic orbits of $X$ lead him to the conclusion that the Lozi map is mixing on the unstable manifold of $X$, ultimately obtaining the following.
\begin{thm}[\cite{strange-attractor-mis, Lozi-likemaps}]\label{misiurewicz-main}
	Let $\Uee_{+}$ be the set of points $(a,b)\in\plane$ fulfilling the following inequalities 
\begin{align*}
	 0<b<1, && 0<a<\dfrac{4-b}{2}, && a>\dfrac{b+2}{\sqrt{2}}.
\end{align*}
	Then the Lozi map $f_{(a,b)}$ has a unique saddle-type fixed point in the right half plane and the closure of the unstable manifold of this point is a strange attractor, on which $f_{(a,b)}$ is topologically mixing.
\end{thm}
We shall extend Theorem \ref{misiurewicz-main} to the orientation preserving case ($b<0$). Similarly to the orientation reversing case, an orientation preserving Lozi map has two fixed points, one of which, $X$, has bounded unstable manifold. A significant difference between those two cases is that in the former, the eigenvalue corresponding to the stable direction of the fixed point is positive, whereas in the latter it is negative. Hence, any orbit on the stable manifold of $X$ accumulates on both sides of it, and in particular one cannot expect to find a homoclinic orbit that would facilitate a construction of a compact attracting neighbourhood of the unstable manifold, that has that manifold as boundary. Instead, we use the unstable manifold of the other fixed point and a heteroclinic orbit to determine the boundary of the attracting region. Note that Cao and Liu \cite{cao-liu-98} propose very similar candidate, but one can compute that their attracting neighbourhood is invariant in much smaller set of parameters, making their choice far from optimal. Apart from the mentioned disadvantage, we closely follow the strategy of proof devised by Misiurewicz. This paper is organized as follows. In Section \ref{sec:prelimiaries} we provide notational conventions and definitions. Section \ref{sec:properties-of-lozi-family} is devoted for basic properties of the Lozi family, such as positions and existence of certain important points. Throughout Section \ref{sec:hyperbolicity} we show that the Lozi family possesses certain hyperbolic properties.\footnote{Note that existence almost everywhere of stable and unstable local manifolds for piecewise affine homeomorphisms was proven by J\'{e}r\^{o}me Buzzi \cite{buzzi}.} We finally use hyperbolicity to prove that Lozi map is mixing on the closure $\cl\ustmani{X}$ of the unstable manifold of the fixed point $X\in\plane$ in Section \ref{sec:main-results}. We finalize our proof of the existence of strange attractors by showing that $\cl\ustmani{X}$ has an attracting neighbourhood of a certain significant form. At this point one can reprove a special case of Theorem 2.3 from \cite{simpson:robust-chaos}, that is we show.
\begin{theorem}\label{thm:mainA}
	Let  \begin{multline*}
	\Uee_{-}=\{(a,b)\in\plane\colon \sqrt{2}<a<2\\\text{ and  }\max\{\frac{1}{8} \left(3 a^2-8 a+ \sqrt{9 a^4-16 a^3}\right),2-\sqrt2 a\} <b<0\}.
	\end{multline*} Then for every $(a,b)\in\Uee_{-}$ the Lozi map $f_{(a,b)}$ has a unique saddle-type fixed point in the lower half plane and the closure of the unstable manifold of this point is a chaotic attractor.
\end{theorem}
Then, we apply renormalization technique to reason that the attractor is maximal in the given neighbourhood. This step is the hardest one and is based on a detailed analysis of a first return map to a certain region of the phase space. We are not aware of any paper that shows maximality or uniqueness of the attractor of the Lozi family. Moreover, it seems that tools used in works mentioned earlier are by far insufficient to tackle such a problem. We show the following.
\begin{theorem}\label{thm:mainB}
	One can find an open parameter region $\Uee_{\text{tan}}\subset\Uee_{-}$ accumulating on $(2,0)$ so that $\cl\ustmani{X}$ is a strange attractor whenever $(a,b)\in \Uee_{\text{tan}}$. Additionally, in $\Uee_{\text{tan}}$ attractors change continuously with respect to the Hausdorff metric.
\end{theorem}
Lastly, in Section \ref{sec:continuity} we show that attractors vary continuously with respect to the Hausdorff metric.
\begin{theorem}\label{thm:mainC}
	In $\Uee_{\text{tan}}$ attractors change continuously with respect to the Hausdorff metric.
\end{theorem}
\begin{figure}[H]
	\caption{Attractor for the parameters $(1.78,-0.5)$.}
	\centering
	\includegraphics[width=\textwidth]{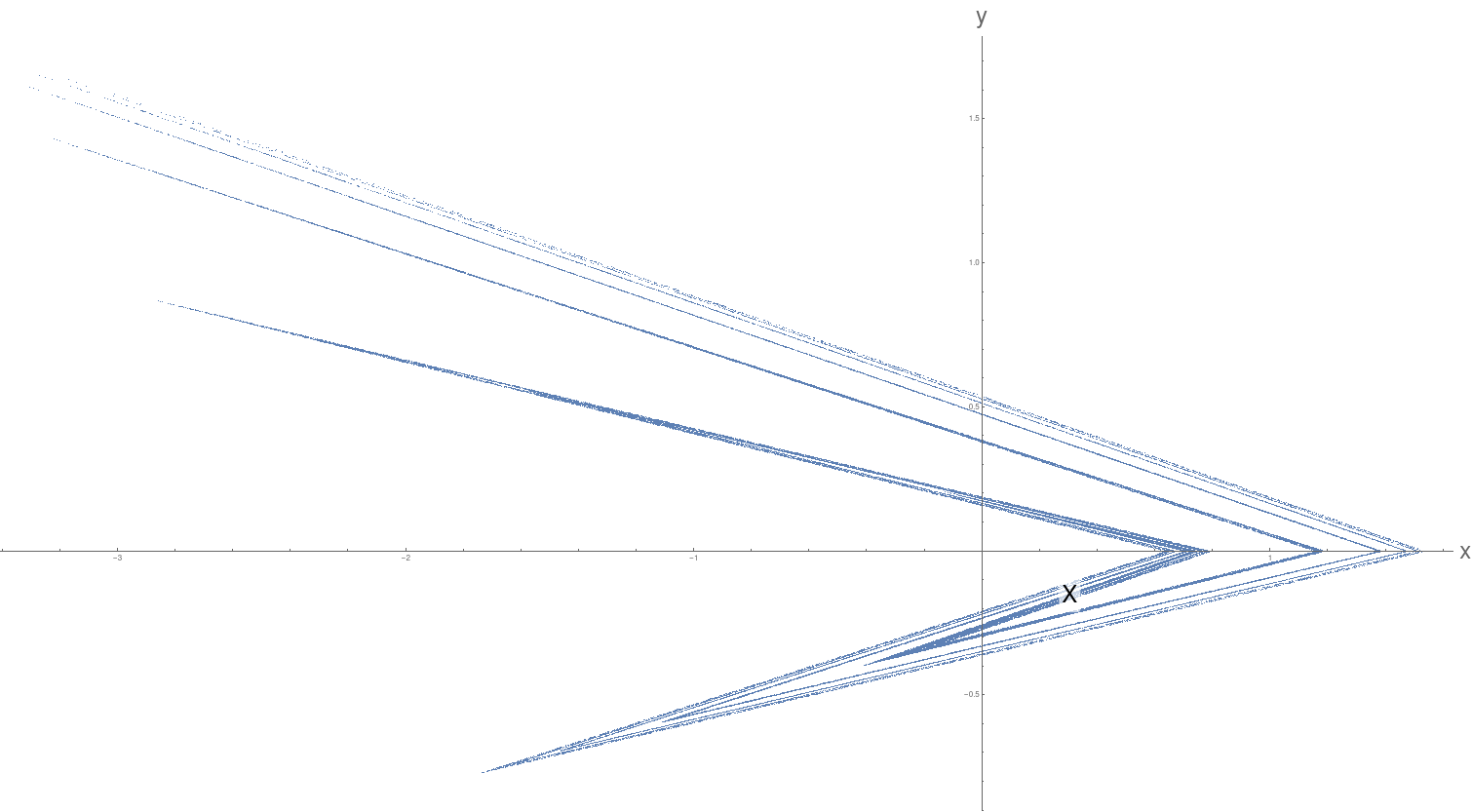}
	\label{fig:attractor}
\end{figure}
\section{Preliminaries}\label{sec:prelimiaries}
\subsection{Notational convention}
Real plane will be denoted by $\plane$. For a set $U\subset \plane$ we denote the interior $\inter U$, boundary $\fr U$ and closure $\cl U=\clu U$. Metric will be denoted by $d$. We will say that $U$ does not separate $\plane$ if $\plane\setminus U$ is connected. Any open (closed) superset $W$ of $U$ will be called an open (closed) neighbourhood of $U$. Directly constructed points will be denoted by capital letters, while points that serve as variables will be denoted by lowercase letters. The set $\nat=\{0,1,...\}$ is a set of natural numbers. Points defined explicitly using parameters will be denoted by capital letters. For a point $p\in \plane$, we let $p_{x}$ and $p_{y}$ be the $x$-coordinate and $y$-coordinate respectively. A line in the real plane $\plane$ defined by the formula $y=ux+v$ will be called positively directed, if $u>0$. For two points $q,r\in \plane$ we will denote by $\overline{q r}$ a segment with $q$, $r$ as endpoints. The Lebesgue measure of $A\subset\plane$ will be denoted by $\lebmes{A}$.
\subsection{Attractors}\label{sec:attractors}
As there is no agreement on the definition of an attractor in literature, to avoid any confusion we state it explicitly. We follow the definition introduced by John Milnor \cite{milnor:on-the-concept-of-attractor} Let $h\colon X\to X$ be a continuous map on a topological space $X$, and $F\subset X$ be a closed subset. We define the realm of attraction $\rho(F)$ of $F$ as all those points $x\in X$ such that $\omega(x)\subset F$. 
\begin{dfn}\label{def:attractor}
	A subset $A\subset X$ will be called an attractor if it satisfies the following\begin{enumerate}[label=(A\arabic*)]
		\item its realm of attraction $\rho(A)$ is of positive Lebesgue measure,
		\item for any closed $A'\subset A$ with $\rho(A')=\rho(A)$ we must have $A'=A$
	\end{enumerate}
\end{dfn}
Some authors require from the attractor additional properties. Nonetheless, we will be working with the above definition, as it allows wide variety of subsets to be attractors and so usually does not discriminate other definitions in use. Main object in this paper is an attractor that exhibits some chaotic properties. In papers on border collision normal form \cite{simpson:robust-chaos} and Lozi-like maps \cite{Lozi-likemaps} relevant to results in this article, there are two definitions of attractor with chaotic properties used. Let us first specify a few notions. An open subset $U\subset X$ will be called a trapping region if $h(\cl U)\subset U$. A compact set $F\subset X$ contained in a trapping region $U$ and satisfying $\bigcap_{n\in\nat}h^{n}(U)=F$ will be called maximal in $U$. A map $h$ is said to be transitive if for any open sets $W,V\subset X$ we can find $i\in\nat$ such that $h^{i}(W)\cap V\neq\emptyset$. On the other hand, if $h^{i}(W)\cap V\neq\emptyset$ holds for almost every $i\in\nat$, then $h$ is said to be mixing. We are ready to state.
\begin{dfn}\label{def:strange-chaotic-att}
	A compact, invariant subset $A\subset X$ will be called a strange attractor if it has a trapping region in which it is maximal. On the other hand, if $A\subset$ is an attractor such that $h|_{A}$ is transitive and periodic points are dense in $A$, $A$ will be called a chaotic attractor and said to have Devaney chaos.
\end{dfn}
\subsection{Hyperbolicity}
Let us recall the definition of the Lozi family. For $a,b\in\real$ we define $f(x,y)=(1+y-a|x|,bx)$. In Section \ref{sec:hyperbolicity} we prove that every Lozi map in a certain subset of the parameter space has a hyperbolic splitting, at points for which it is differentiable at every point of their orbit. It will follow that almost everywhere one can find a stable and unstable local linear manifolds, which in turn can be extended to manifolds of infinite lengths. These manifolds are broken segments, hence they are not differentiable at every point. Let us introduce the necessary definitions.
Following \cite{strange-attractor-mis}, for a homeomorphism $g$ of a metric space $M$ onto itself we define a local stable  (respectively, unstable) manifold at a point $q$ as 
\begin{multline*}
 W^{s}_{q,\epsilon,loc}=\\\{r\in M:\lim_{n\to\infty}d(g^{n}(q),g^{n}(r))=0\text{ and }d(g^{i}(q),g^{i}(r))<\epsilon\text{ for every } i\in\nat\} 
\end{multline*}and
\begin{multline*}
W^{u}_{q,\epsilon,loc}=\\\{r\in M:\lim_{n\to\infty}d(g^{-n}(q),g^{-n}(r))=0\text{ and }d(g^{-i}(q),g^{-i}(r))<\epsilon\text{ for every } i\in\nat\} 
\end{multline*}for some $\epsilon>0$. Whenever we will say that a local stable (unstable) manifold with certain properties exists, we will assume that one can find $\epsilon>0$ for which the manifold exists, dropping usually $\epsilon$ in notation. In this way, by a local linear stable (unstable) manifold it is meant a stable (unstable) manifold with $\epsilon>0$ for which the manifold is a segment. The global stable (unstable) manifold at a point $q$ is defined as $W_{q}^{s}=\bigcup_{n\in\nat}g^{-n}(W^{s}_{g^{n}(q),loc})$ ($W_{q}^{u}=\bigcup_{n\in\nat}g^{n}(W^{u}_{g^{-n}(q),loc})$). Note that global manifolds might not be an embedded manifolds in general. Let us recall that the $\omega$-limit set $\omegalim{S}$ of a set $S$ is defined as \[\omega(S)=\{x\in X\colon x=\lim_{i\to\infty}g^{n_{i}}(x_{i})\text{ for some sequences }\{n_{i}\}_{i\in\nat}\in \nat_{0}^{\nat}\text{ and }\{x_{i}\}_{i\in\nat}\in X^{\nat}\}.\]

\subsection{The parameter set}
As in \cite{strange-attractor-mis} we will consider various assumptions, which we shall list below, whereas not in every lemma every assumption is needed, and some of them are stronger than others. The functions in \ref{c3} and \ref{c6} are considered whenever they are defined in $\real$.
\begin{enumerate}[label=(C\arabic*)]
	\item \label{c-1} $ 1<a\leq2 $, $ -1<b<0 $ and $ a+b > 1 $,
	\item \label{c2} $1<a\leq 2\text{ and }\max\{1-a,(1-2a)/4\}< b\leq 0$,
	\item \label{c3} $1<a< 2\text{ and }\max\{1-a,(1-2a)/4,\frac{1}{8} \left(3 a^2-8 a+ \sqrt{9 a^4-16 a^3}\right) \}< b< 0$,
	\item \label{c5} $\sqrt{2}<a<2$ and $2-\sqrt2 a <b<0 $,
	\item \label{c6} $\sqrt{2}<a<2$ and $\max\{\frac{1}{8} \left(3 a^2-8 a+ \sqrt{9 a^4-16 a^3}\right),2-\sqrt2 a\} <b<0 .$
\end{enumerate}
Let us investigate how bounds on the parameter $b$ correlate with each other. Put $p_{1}(a)=1-a$, $p_{2}(a)=\frac{1}{4} (1-2 a)$, $p_{3}(a)=\frac{1}{8} \left(3 a^2-8 a+ \sqrt{9 a^4-16 a^3}\right)$ and $p_{4}(a)=2-\sqrt{2}a$. Relations of $p_{1},$ $p_{2}$,  $p_{3}$ and $p_{4}$ are shown on the figure \ref{fig:final} and can be summarised in the following fashion
\begin{center}
	\begin{tabular}{ c|c } 
		$ 1<a< \frac{3}{2}$ & $p_{1}(a)>p_{2}(a)$  \\
		\hline
		$ \frac{3}{2}<a<\frac{16}{9}$ & $p_{2}(a)>p_{1}(a)$  \\ 
		\hline
		$ \frac{16}{9}< a<2$ & $p_{3}(a)>p_{2}(a)>p_{1}(a)$  \\ 
		\hline
		$  \sqrt{2}<a<2$ & $\max\{p_{3}(a),p_{4}(a)\}>\max\{p_{1}(a),p_{2}(a)\}$  \\ 
	\end{tabular}
\end{center}
The set satisfying \ref{c6} is depicted as shaded region on Figure \ref{fig:final}. The two dashed lines are $a=3/2$ and $a=16/9$.
\begin{figure}[h]
	\caption{Bounds on the parameter $b$}
	\centering
	\includegraphics[width=\textwidth]{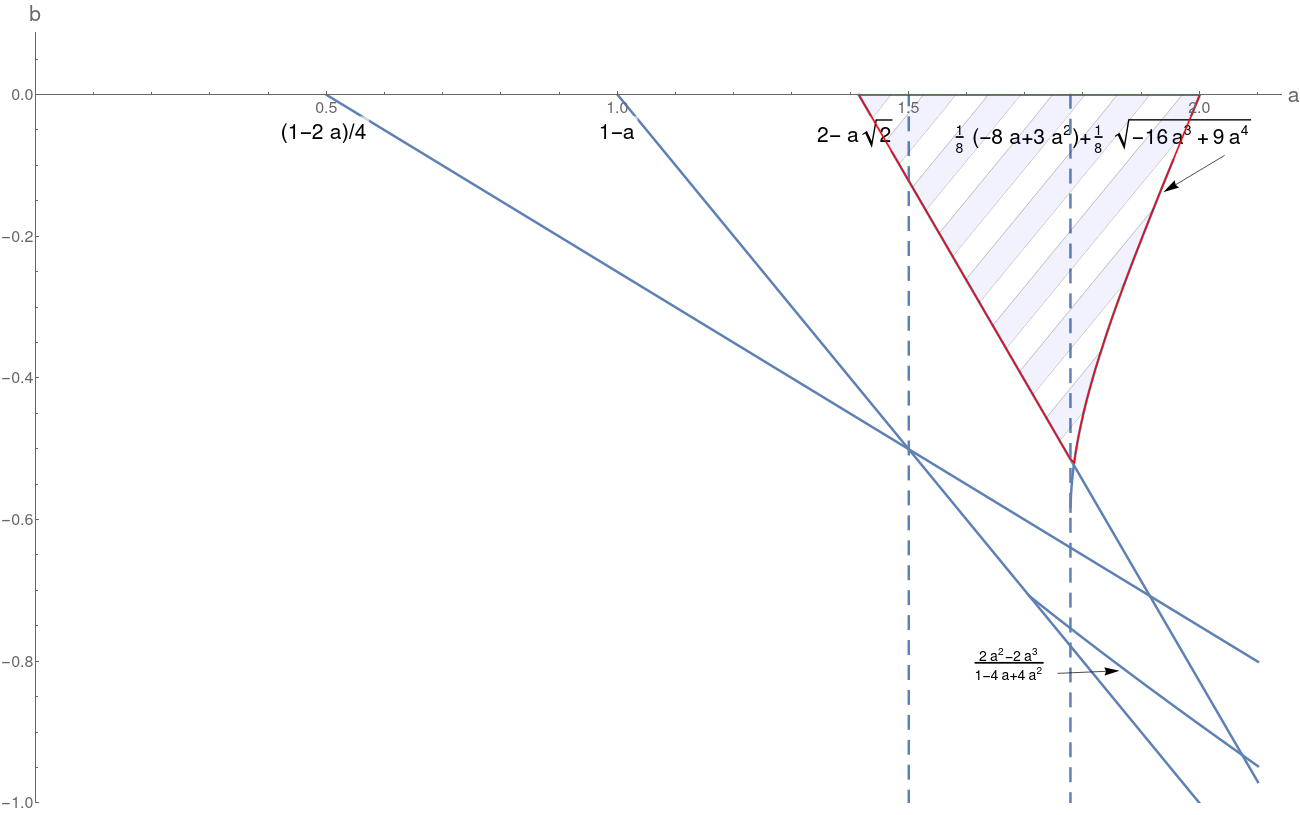}
	\label{fig:final}
\end{figure}

\section{Basic properties of the Lozi family}\label{sec:properties-of-lozi-family}
Throughout the rest of this paper we will often refer to properties and positions of points established in this section. It will be helpful to commit Figure \ref{fig:trap-reg} to memory. Let us assume \ref{c-1}. Since $-\frac{a^2}{4}<1-a< b$ for $1<a\leq2$, \ref{c-1} implies $-\frac{a^2}{4}<b\leq0$. Lozi map has two fixed points $X=(1/(1+a-b),b/(1+a-b))$ and $Y=(-1/(a+b-1),-b/(a+b-1))$, located in the second and fourth quadrant. They are hyperbolic with eigenvalues of $Df$ being \[ \alpha~\text{ and  }\beta\text{ at }Y,\]
 and 
\[ -\alpha ~\text{ and } -\beta \text{ at }X, \]where $\beta=\dfrac{a+\sqrt{a^2 +4b}}{2}$ and $\alpha=-\dfrac{b}{\beta}=\dfrac{a-\sqrt{a^2 +4b}}{2}$. Note that $$0<\alpha<1<\beta$$ for $(a,b)$ satisfying \ref{c-1}. The eigenvector corresponding to an eigenvalue $\lambda$ is $(\lambda,b).$ Draw a line, starting at $Y$ in the direction of the unstable manifold of $Y$, that is $(\beta,b)$. It will intersect the $y$-axis at the point $\xi$ in the upper half plane and the $x$-axis at the point $A=\left(Y_{x}(1-\beta),0\right)$ in the right half plane, note that $f(\xi)=A$. Denote $B=f(A)$ and note that $B$ lies in the lower half plane. Let $M$ be the point of intersection of the $y$-axis with the line containing $\seg{AB}$. 
\begin{figure}[h]
	\caption{Triangle $G=\Delta YAZ$ and $f(G)$.}
	\centering
	\includegraphics[width=\textwidth]{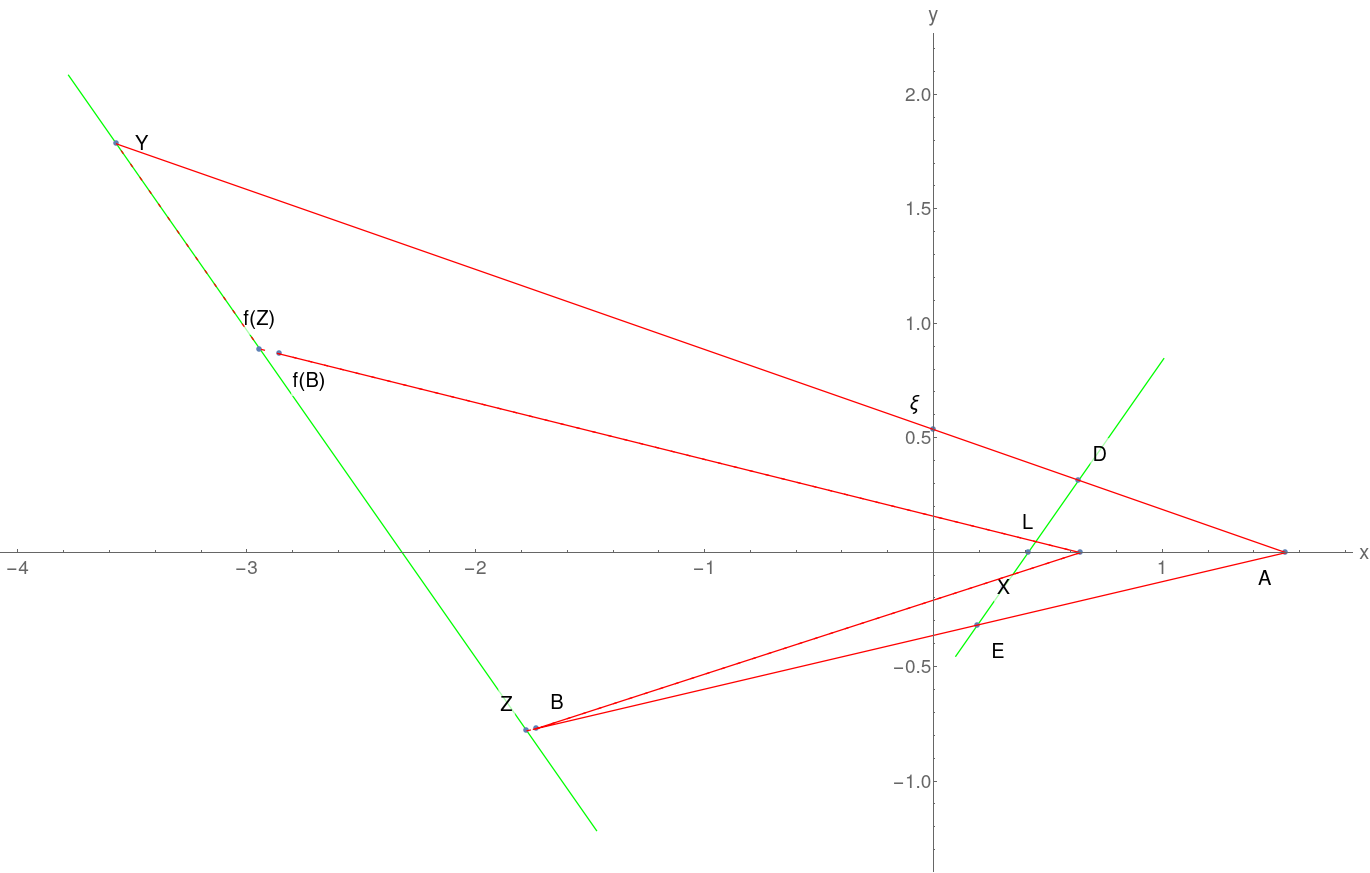}
	\label{fig:trap-reg}
\end{figure}

Let us define $L=\left(X_{x}(1+\alpha),0\right)$ to be the point of intersection of the line containing the local stable manifold of $X$ with the $x$-axis. Let us first compute $L_{x}-A_{x}$. After some transformations, we arrive at\[ L_{x}-A_{x}=\frac{a \left(1-\beta\right)}{(a+b-1) \left(\beta+1\right)}, \]which is negative if we assume \ref{c-1}. Hence, $ L_{x}<A_{x}$. Now, if $X$ is positioned below the segment $\seg{AB}$, then $\seg{AB}$ and $\seg{XL}$ intersect at the point $\tilde E$ above $X$ on $\ustlocmani{X}$. Note that $f(\tilde E)$ lies below $X$ on $\ustlocmani{X}$, which implies that $f(\seg{AM})$ lies below $\seg{AB}$. Consequently, $f(M)_{x}>A_{x}$, which contradicts that $y\mapsto f(0,y)_{x}$ is increasing. We have proved that $X$ lies above $\seg{AB}$. 

Moreover, under condition \ref{c-1} $B$ lies in the lower half-plane and consequently so does $M$. Suppose that $f(M)_{x}>L_{x}$. Then there exists the point $E$ of intersection of $\stlocmani{X}$ with $\seg{A B}$, as it is the common point of $\seg{f^{-1}(L) L}$, which connects the $x$-axis with the $y$-axis, and $\seg{A B}$. We also infer that $D=f^{-1}(E)$ is the intersection of $\stmani{X}$ and $\seg{A B}$. Moreover, $X\in\seg{ED}$, since the sign of the eigenvalue of $X$, corresponding to $\stmani{X}$ is negative, hence $E$ and $D$ must lie on the opposite sides of $X$. We can therefore define the triangle $H_{0}=\Delta ADE$. As $f(M)_{x}>L_{x}$, we have $f(M)\in H_{0}$. It also shows that $D_{x}>0$ and $E_{x}>0$. Indeed, if $X\in\seg{ED}$, then $E_{x}<0$ is only possible if $M_{y}<f^{-1}(L)_{y}$, but then we also have $f(M)_{x}<L_{x}$. Inequality $D_{x}>0$ is a consequence of the fact, that the stable eigenvector of $X$ has both coordinates either positive or negative. 
\begin{rem}\label{substitution}
	Let $a,b$ satisfy \ref{c-1}. We are going to use the change of variables $b=b_{a}(r)=(r^2-2ra)/4$ for $0\leq r\leq a$. For every $a>0$, it defines a decreasing bijection $$[0,a]\ni r\mapsto b_{a}(r)\in[-a^2/4,0]$$ with the inverse $b_{a}^{-1}(b)=a-\sqrt{a^2+4 b}$. Condition \ref{c-1} is then equivalent to 
	\begin{equation}\label{c1-varchange}
	 0< r<2~\text{ and }~ \frac{r+2}{2}<a\leq 2. \tag{C1$_{a,r}$}
	\end{equation}
\end{rem}

Now we find conditions under which $f(M)_{x}>L_{x}$ holds. By substituting $b=(r^2-2ra)/4$, for $(a,r)$ satisfying \eqref{c1-varchange}, we get
\[ f(M)_{x}=1+\dfrac{bA_{x}^{2}}{(a+1)A_{x}-1}=1-\dfrac{(2a-r)r}{(2a+r)(2-r)} \]and
\[ L_{x}=X_{x}\left(1+\alpha\right)= \dfrac{2}{2+2a-r}.\]Then we have
\[ f(M)_{x}-L_{x}=\frac{4 a (1-r) (2 a-r)}{(2-r) (2 a-r+2) (2 a+r)}. \]

It is positive if we additionally assume $0< r <1$. Hence, reversing variable change, inequality $f(M)_{x}>L_{x}$ holds for 
\begin{equation*}
1<a\leq 2\text{ and }\max\{1-a,(1-2a)/4\}< b< 0,
\end{equation*}which is condition \ref{c2}.
We can summarize 
\begin{lem}\label{f(M)in H0}
	If condition \ref{c2} is satisfied the triangle $H_{0}$ exists, is positioned in the right half plane and contains $f(M)$ in the interior and $X$ in the boundary.
\end{lem}
We will need the following
\begin{lem}\label{BliesaboveWsY}
	If \ref{c3} is satisfied, then $B$ lies above the line of the stable manifold of $Y$ in the left half plane.
\end{lem}
\begin{proof}
	Assume \ref{c3} and note that it is stronger then \ref{c-1}, which by Remark \ref{substitution} is equivalent to \eqref{c1-varchange}. We compute conditions under which $B$ lies above the line of the stable manifold of $Y$, that is line given by \[ 
	\det\begin{pmatrix}
	x-Y_{x} & \alpha\\
	y-Y_{y} & b
	\end{pmatrix}=0,
	 \]or
	\begin{equation}\label{B1}
	\dfrac{1}{\beta}(y-Y_{y})+x-Y_{x}=0.
	\end{equation}
	Substituting coordinates of $B$ to the left hand side we get 
	
	\begin{equation}\label{subB}
	\frac{2 b-a \left(\beta-2\right)}{a+b-1}=0.
	\end{equation}
	 Since coefficient of $y$ in \eqref{B1} is positive, if \eqref{subB} is positive, $B$ lies above the line of the stable manifold of $Y$. Substituting  $b=(r^2-2ra)/4$ for $(a,r)$ satisfying \eqref{c1-varchange} and $r\in[0,a]$, we obtain 
	\begin{equation*}\label{k1}
	\frac{4 a^2+2 a (r-4)-2 r^2}{(r-2) (2 a-r-2)}=0.
	\end{equation*} Assuming \eqref{c1-varchange}, the left hand side is positive if $2 a^2+a (r-4)-r^2<0$. Hence it is satisfied on the set of $(a,r)$ fulfilling one of the following conditions
	\begin{enumerate}[label=(\roman*)]
		\item\label{r1} $1<a<\frac{16}{9}\text{ and  } 0<r<2 a-2$ 
		\item\label{r2} $\frac{16}{9}\leq a<2\text{ and  } \left(0<r<\frac{a}{2}-\frac{1}{2} \sqrt{9 a^2-16 a}\text{ or  } \frac{1}{2} \sqrt{9 a^2-16 a}+\frac{a}{2}<r<2 a-2\right).$ 
	\end{enumerate}If we reverse the variable change via $r\mapsto a-\sqrt{a^2+4b}$ we obtain
	\begin{enumerate}[label=(\arabic*)]
		\item $ 1<a<\frac{16}{9}\text{ and  } 1-a<b<0$,
		\item $ \frac{16}{9}\leq a<2$ and $ \frac{1}{8} \left(3 a^2-8 a\right)+\frac{1}{8} \sqrt{9 a^4-16 a^3}<b<0$ or \\$ 1-a<b<\frac{1}{8} \left(3 a^2-8 a\right)-\frac{1}{8} \sqrt{9 a^4-16 a^3}).$
	\end{enumerate} which is weaker then condition \ref{c3}.
\end{proof}

We assume \ref{c3}. Let $ Z $ be the point of intersection of the stable manifold of $Y$ with the line containing $ \seg{AB} $. Denote by $G$ the closed triangle $\Delta YAZ$. The following lemma was proved in a more general setting, for border collision normal form, in \cite{simpson:devaney-chaos}.
\begin{lem}\label{DeltaYAZisINVARIANT}
	Assuming \ref{c3}, the triangle $G$ is forward-invariant under $f$.
\end{lem}
\begin{proof}
	Note that $\seg{YZ}\subset\stmani{Y}$ and the eigenvalues at $Y$ are positive, hence $f(Z)\in G$. Under the condition \ref{c-1} $B$ lies in the third quadrant, so it must be mapped by $f$ to the upper half-plane. Moreover, as $f$ is linear in the left half plane and expanding eigenvalue at $Y$ is positive, by Lemma \ref{BliesaboveWsY} $f(B)$ must be located above $\seg{YZ}$. Lastly, $\seg{f(B)f(M)}$ cannot cross $\seg{YA}$, as unstable manifolds have no self intersections. Consequently, $f(B)$ is contained in $G$.
\end{proof}

\section{Hyperbolicity}\label{sec:hyperbolicity}
Following Misiurewicz, we will prove that Lozi map is in a certain way hyperbolic; i. e. at almost all points there exists a hyperbolic splitting. To cope with the case of negative $b$ we will follow closely Misiurewicz's proof of the following theorem. Put $c=(a-\sqrt{a^2-4b})/2.$ Let us denote by $T_{z}$ the tangent space at $z\in\plane$.
\begin{thm}[see {\cite[Thm. 3.]{strange-attractor-mis}}]Let $0<b<1$, $a>0$ and $a-b>1$. Then
	\begin{enumerate}
		\item If $z\in \real^2$ is such that $f$ is differentiable at all points $f^n(z)$, $n=0,1,2,...,$ then there exists a one-dimensional subspace $E^{s}_{z}\subset T_{z}$ such that
		\begin{equation}\label{c^n}
		\norma{Df^n(v)}\leq c^n \norma{v}\text{ for every }v\in E_{z}^{s}.
		\end{equation}Moreover, $Df_{z}(E^{s}_{z})=E^{s}_{f(z)}.$
		\item If $z\in \real^2$ is such that $f^{-1}$ is differentiable at all points $f^{-n}(z)$, $n=0,1,2,...,$ then there exists a one-dimensional subspace $E^{u}_{z}\subset T_{z}$ such that
		\begin{equation}\label{b/c^n}
		\norma{Df^n(v)}\geq \Big(\dfrac{b}{c}\Big)^n \norma{v}\text{ for every }v\in Df^{-n}(E_{z}^{u}).
		\end{equation}Moreover, $Df^{-1}_{z}(E^{u}_{z})=E^{u}_{f^{-1}(z)}.$
		\item if $z\in \real^2$ is such that $f$ is differentiable at all points $f^{n}(z)$, $n=0,\pm 1,\pm 2,...,$then $T_{z}=E^{s}_{z}\oplus E^{u}_{z}$. 
	\end{enumerate}
\end{thm}

Let us now consider the case of $b<0$. Note that, as Misiurewicz mentions in his work, constants $c=(a-\sqrt{a^2-4b})/2$ and $b/c=(a+\sqrt{a^2-4b})/2$ are eigenvalues of a two-periodic point. On the other hand, the eigenvalues of one of the fixed points, that is $\alpha=\alpha_{b}=(a-\sqrt{a^2+4b})/2$ and $\beta=\beta_{b}=(a+\sqrt{a^2+4b})/2$ seem to work as well. It does not matter which point we take, if we additionally take the absolute value of eigenvalues. As it turns out, $\alpha_{-b}=c$ and $\beta_{-b}=b/c$ for $b>0$. Hence, given our negative $b$, we can replace in the reasoning of Misiurewicz's proof $b$ by $-b$. For completeness we enclose the proof below.

First, we need to find invariant families of stable and unstable cones. The proof of their existence can be also found in \cite{rychlik:invariant-measures}, although the author used different methods. We present a proof based on \cite{Lozi-likemaps}. For $z=(x,y)\in\plane$, let us define
\[
C^{s}_{z}:=\left\{\binom{t}{r}\in T_{z}~:~\modul{t}\leq \dfrac{1}{\beta} |r|\right\},
\] and
\[
C^{u}_{z}:=\left\{\binom{t}{r}\in T_{z}~:~|r|\leq \alpha |t|\right\}.
\]As the following holds, $C^{s}_{z}$ and $C^{u}_{z}$ will be called respectively stable and unstable cones at $z$. 
\begin{lem}\label{inv-cones}
	Assume \ref{c-1}. Let $z=(x,y)\in\plane$ be such that $x\neq 0$, then $Df^{-1}_{f(z)}$ ($Df_{z}$)  maps the stable cone $ C^{s}_{f(z)} $ (unstable cone $C^{u}_{z}$) into $ C^{s}_{z} $ ($ C^{u}_{f(z)} $) and expands all its vectors by a factor of at least $1/\alpha$ ($\beta$).
\end{lem}
\begin{proof}
	Note that if $x\neq 0$, then $f(z)_{y}\neq 0$ and $Df^{-1}_{f(z)}$ exists. Let $z=(x,y)\in \plane$ be such that $x\neq 0$. We can compute
	\[ 
	Df_{z}=\begin{pmatrix}
	-\sigma a & 1\\
	b & 0
	\end{pmatrix},
	\] where $\sigma=\sgn x$ is the sign of $x$. Let
	\[ 
	Df_{z}
	\binom{t}{r}=\binom{\bar{t}}{\bar{r}}\in T_{f(z)},
	\]for $(t,r)\in T_{z}$. Then \[ 
	Df^{-1}_{f(z)}
	\binom{\bar{t}}{\bar{r}}=\binom{t}{r}.
	\]
	It is enough to show that the following holds
	\begin{enumerate}[label=(\arabic*)]
		\item\label{app2} if $|\bar{t}|\leq \dfrac{1}{\beta}|\bar{r}|$, then $|t|\leq \dfrac{1}{\beta}|r|$, $|\bar{t}|\leq \alpha|t|$ and $|\bar{r}|\leq \alpha|r|$,
		\item\label{app3} if $|r|\leq \alpha|t|$, then $|\bar{r}|\leq \alpha|\bar{t}|$, $|\bar{r}|\geq \beta|r|$ and $|\bar{t}|\geq \beta|t|,$
	\end{enumerate}
	Both $\alpha$ and $\beta=|b|/\alpha$ are roots of $\lambda^2-a\lambda+\modul b=0$, so $\alpha+\beta=a$. We may assume that $\bar{r}=1$ and $|\bar{t}|\leq \alpha/|b|$. Then $t=1/b$ and $r=\bar{t}-\sigma a/b$, and we have to prove that 
	\begin{equation}\label{app1}
	1\leq \alpha \Big| \bar{t}-\dfrac{\sigma a}{b}\Big|\text{, }\dfrac{\alpha}{|b|}\leq \alpha\cdot \dfrac{1}{|b|}\text{ and }1\leq \alpha \Big|\bar{t}-\dfrac{\sigma a}{b}\Big|
	\end{equation}
	The second inequality holds and the first and third are identical. We have 
	\[ \Big| \bar{t}-\dfrac{\sigma a}{b}\Big|\geq\dfrac{a}{|b|}-|\bar{t}|\geq \dfrac{a}{|b|}-\dfrac{\alpha}{|b|}=\dfrac{1}{\alpha},
	\]so \eqref{app1} holds.
	In a similar way we prove \ref{app3}, we may assume that $t=1$ and $|r|\leq \alpha$. Then $\bar{t}=\sigma a+r$ and $\bar{r}=b$. It immediately follows that $\beta|r|\leq \beta\alpha=|b|=|\bar{r}|$. It remains to prove second and fourth inequality in \ref{app3}. That is
	\begin{equation*}
	|b|\leq \alpha |\sigma a+r|\text{ and }|\sigma a+r|\geq \dfrac{|b|}{\alpha}.
	\end{equation*} As in the case \ref{app2}, they are equivalent. We have
	\[ |\sigma a+r|\geq a-\alpha=\beta=\dfrac{|b|}{\alpha},
	\]which concludes our proof.
\end{proof}
\begin{thm}\label{hyp}Let us assume \ref{c-1}. Then
	\begin{enumerate}[label=(\roman*)]
		\item\label{hyp1} If $q\in \real^2$ is such that $f$ is differentiable at all points $f^n(q)$, $n=0,1,2,...,$ then there exists a one-dimensional subspace $E^{s}_{q}\subset T_{q}$ such that
		\begin{equation}\label{c^nb<0}
		\norma{Df^n(v)}\leq \alpha^n \norma{v}\text{ for every }v\in E_{q}^{s}.
		\end{equation}Moreover, $Df_{q}(E^{s}_{q})=E^{s}_{f(q)}.$
		\item\label{hyp2} If $r\in \real^2$ is such that $f^{-1}$ is differentiable at all points $f^{-n}(r)$, $n=0,1,2,...,$ then there exists a one-dimensional subspace $E^{u}_{r}\subset T_{r}$ such that
		\begin{equation}\label{b/c^nb<0}
		\norma{Df^n(v)}\geq \beta^n \norma{v}\text{ for every }v\in Df^{-n}(E_{r}^{u}).
		\end{equation}Moreover, $Df^{-1}_{r}(E^{u}_{r})=E^{u}_{f^{-1}(r)}.$
		\item\label{hyp3} if $r\in \real^2$ is such that $f$ is differentiable at all points $f^{n}(r)$, $n=0,\pm 1,\pm 2,...,$then $T_{r}=E^{s}_{r}\oplus E^{u}_{r}$.
	\end{enumerate}
\end{thm}
\begin{proof}
We are ready to construct the stable and unstable subspaces of the tangent space. Take points $q$ and $r$ such that $f$ is differentiable at $f^{n}(q)$ and $f^{-1}$ is differentiable at $f^{-n}(r)$, $n\in\nat$. Therefore, $f^{n-1}(q)_{x}\neq 0$ for $n\geq 1$ and we can apply Lemma \ref{inv-cones} to the sequence $\{f^{n-1}(q)\}_{n\geq 1}$. We then obtain $Df_{f(f^{n-1}(q))}^{-1}=Df_{f^{n}(q)}^{-1}$ maps the stable cone $C^{s}_{f^{n}(q)}$ into $C^{s}_{f^{n-1}(q)}$ and expands all its vectors by a factor of at least $1/\alpha$. As $f^{-1}$ is differentiable at $f^{-n}(r)$ for every $n\geq 0$, we have $f^{-n}(r)_{y}\neq0$ and so $f^{-n-1}(r)_{x}\neq0$. We can then apply Lemma \ref{inv-cones} to the sequence $\{f^{-n}(r)\}_{n\geq1}$. We learn that $Df_{f^{-n}(r)}$ maps the unstable cone $C^{u}_{f^{-n}(r)}$ into $C^{u}_{f^{-n+1}(r)}$ and expands all its vectors by a factor of at least $\beta$. Note that both $1/\alpha$ and $\beta$ are larger then 1. Define
\[ 
E^{s}_{q}=\bigcap_{n=0}^{\infty}Df^{-n}_{f^{n}(q)}(C^{s}_{f^{n}(q)}),
\]
\[ 
 E^{u}_{r}=\bigcap_{n=0}^{\infty}Df^{n}_{f^{-n}(r)}(C^{u}_{f^{-n}(r)}).
\]The inequalities \eqref{c^nb<0} and \eqref{b/c^nb<0} follow from the properties of stable and unstable cones. The sets $E^{s}_{q}$ and $E^{u}_{r}$ are intersections of descending sequences of non-empty closed cones, therefore, they both contain one-dimensional spaces. We have 
\[ E^{s}_{q}\subset J^{s}_{q}=\{v\in T_{q}:~\lim_{n\to\infty}\norma{Df^{n}(v)}=0\}\text{ and  } \] \[ E^{u}_{r}\subset J^{u}_{r}=\{v\in T_{r}:~\lim_{n\to\infty}\norma{Df^{-n}(v)}=0\}. \] Moreover, $J^{s}_{q}$ and $J^{u}_{r}$ are vector spaces. Hence, in order to show that $E^{s}_{q}$ and $E^{u}_{r}$ are one dimensional spaces, it is enough to show that the sets $T_{q}\setminus J^{s}_{q}$ and $T_{r}\setminus J^{u}_{r}$ are non-empty. But they contain interiors of unstable and stable cones, respectively. Since $Df_{q}$ and $Df_{r}^{-1}$ are non-degenerate and the spaces $E^{s}_{q}$, $E^{s}_{f(q)}$, $E^{u}_{r}$ and $E^{u}_{f^{-1}(r)}$ are one-dimensional, we have $Df_{q}(E^{s}_{q})=E^{s}_{f(q)}$ and $Df_{r}^{-1}(E^{u}_{r})=E^{u}_{f^{-1}(r)}$. This completes the proof of statements \ref{hyp1} and \ref{hyp2}.

Statement \ref{hyp3} follows from the fact that the intersection of stable and unstable cones is 
\[ \Bigg\{\binom{0}{0}\Bigg\}. \qedhere\]\end{proof}

Since $f$ is linear in the left and right half-planes and $f^{-1}$ is linear in the lower and upper half-planes, the following proposition follows immediately from Theorem \ref{hyp}.

\begin{prop}[see {\cite[Prop. 3.]{strange-attractor-mis}}]\label{proposition 3}
	Let $a$, $b$ satisfy \ref{c-1}. Then, for a point $q\in \real^{2}$, we have 
	\begin{enumerate}[label=(\roman*)]
		\item if $d(f^{n}(q), \text{vertical axis})\geq \theta \alpha^{n}$ for some $\theta>0$ and all $n\in\nat$, then the segment $\{q+tv^{s}_{q}:~|t|<\theta \}$ (where $v^{s}_{q}$ is a unit vector in the stable direction at $q$) is a local stable manifold at $q$,
		\item if $d(f^{-n}(q), \text{horizontal axis})\geq \theta /\beta^{n}$ for some $\theta>0$ and all $n\in\nat$, then the segment $\{q+tv^{u}_{q}:~|t|<\theta \}$ (where $v^{u}_{q}$ is a unit vector in the unstable direction at $q$) is a local unstable manifold at $q$.
	\end{enumerate}
\end{prop}
Let $K$ be the region to the right of the line passing through $Y$ in its stable direction; i. e.
\begin{equation}\label{K}
K=\Big\{(x,y):~	y-Y_{y}+\beta(x-Y_{x})\geq 0\Big\}.
\end{equation}

\begin{lem}[see {\cite[Lem. 2.]{strange-attractor-mis}}]\label{lemma 2}
	Let $a$, $b$ satisfy \ref{c3}, then 
	\begin{enumerate}[label=(\roman*)]
			\item \label{invariance}$ f^{-1}(K)\subset K $
			\item \label{intersection}there exists a constant $\gamma>0$ such that, for every $y\leq 1$, the intersection of the horizontal line $\real\times\{y\}$ with the set $ f^{-1}(K) $ has length bounded by $\gamma$
			\item \label{subset}$ G\subset K $
	\end{enumerate}
\end{lem}
\begin{proof}
	We can compute
	\[
	f^{-1}(K)=\left\{(x,y):~y\geq \dfrac{a\beta\sgn x-b}{\beta}x+Y_{x}-1+\dfrac{Y_{y}}{\beta}\right\}.
	\]Using that $\alpha+\beta=a$ and $\alpha\beta=-b,$ we learn that the boundary of $f^{-1}(K)$ in the left half plane coincides with that of $K$ in the left half plane, on the other hand for  $x>0$ the boundary of $f^{-1}(K)$ is a positively directed half line. Hence the region $\Delta=f^{-1}(K)\cap\{(x,y):~y\leq1\}$ is a triangle, $f^{-1}(K)\subset K$ and \ref{intersection} holds with $\gamma=\diam \Delta$. 
	
	Condition \ref{subset} follows from Lemma \ref{BliesaboveWsY}.
\end{proof}

Denote by $K_{\nu}^{s}$ ($K_{\nu}^{u}$) the set of these points  $q\in G$ at which the segment $\{q+tv_{q}^{s}:~|t|<\nu\}$ ($\{q+tv_{q}^{u}:~|t|<\nu\}$) is not a local stable (unstable) manifold at  $q$.
\begin{lem}[see {\cite[Lem. 3.]{strange-attractor-mis}}]\label{lemma 3}
	Let $a$, $b$ satisfy \ref{c3}, then there exists a constant $\delta>0$ such that $\lebmes{K_{\nu}^{s}}\leq\delta\nu$ and $\lebmes{K_{\nu}^{u}}\leq\delta\nu$ for all $\nu>0$.
\end{lem}
\begin{proof}
	Follows as usual without any unobvious changes. We only have to use hyperbolic constant for the case of $b<0$, that is we have to substitute $\textbf{b}\rightsquigarrow|b|=-b$ as before. For clarity let us write everything down.
	
	By Proposition \ref{proposition 3}, we have $K^{s}_{\nu}\subset\bigcup_{n\in\nat}f^{-n}(L^{s}_{\nu,n})\cap G$ and $K^{u}_{\nu}\subset\bigcup_{n\in\nat}f^{n}(L^{u}_{\nu,n})\cap G$, where $L^{s}_{\nu, n}=\{(x,y)\in\real^{2}:~|x|<\nu\alpha^{n}\}$, $L^{u}_{\nu, n}=\{(x,y)\in\real^{2}:~|x|<\nu/\beta^{n}\}$.
	
	Since $G$ is bounded, there exists a constant $\phi>0$ such that $\lebmes{L^{s}_{\nu, n}\cap G}\leq \nu\phi\alpha^{n}$ for all $\nu$ and $n$. The absolute value of the Jacobian of $f^{-1}$ is equal to $1/|b|$ and we have $f(G)\subset G$. Therefore, $\lebmes{f^{-n}(L^{s}_{\nu, n})\cap G}\leq \lebmes{f^{-n}(L^{s}_{\nu, n}\cap G)}\leq\nu\phi/\beta^{n}$. We have $1/\beta=\alpha/|b|<1$, hence $\lebmes{K^{s}_{\nu}}\leq\nu\phi \dfrac{1}{1-1/\beta}$.
	
	By Lemma \ref{lemma 2} we have 
	\[ 	f(G)\subset G\subset K. \]
 	Hence, $G\subset f^{-1}(K)$. Again by Lemma \ref{lemma 2}, we get $f^{-n}(G)\subset f^{-1}(K)$ for all $n\geq0$. Consequently, $\lebmes{f^{n}(L^{u}_{\nu, n}}\cap G)=|b|^{n}\lebmes{L^{u}_{\nu, n}\cap f^{-n}(G)}\leq\nu\gamma\alpha^{n} $ for all $n\in\nat$ and $0<\nu\leq 1$. Hence, $\lebmes{K^{u}_{\nu}}\leq \nu\gamma/(1-\alpha)$ for all $0<\nu\leq 1$.
	
	Now, we set $\delta=\max\{\phi /(1-1/\beta),\gamma /(1-\alpha),m(G)\}$, and we get $\lebmes{K^{u}_{\nu}}\leq\delta\nu$ and $\lebmes{K^{s}_{\nu}}\leq\delta\nu$ for all $\nu>0$.
\end{proof}
\begin{thm}[see {\cite[Thm. 4.]{strange-attractor-mis}}]\label{hyperbolicity}
	Let $a$, $b$ satisfy \ref{c3}. Then at almost all $q\in G$, there exist linear local stable and unstable manifolds and broken linear global stable and unstable manifolds of infinite length.
\end{thm}
\begin{proof}
	By the definition of the sets $K^{s}_{\nu}$ and $K^{u}_{\nu}$, linear local stable and unstable manifolds exist for all $q\in \tilde{G}:=G\setminus(\bigcap_{\nu>0}K^{s}_{\nu}\cup \bigcap_{\nu>0}K^{u}_{\nu})$. By Lemma \ref{lemma 3} these manifolds exist almost everywhere in $G$.
	
	A broken linear stable (unstable) manifold at $q\in \tilde G$ of length at least $\nu$ can be obtained by first taking a local linear stable (unstable) manifold of length $\nu\alpha^{k}$ ($ \nu/\beta^{k} $) at $f^{k}(q)$ ($ f^{-k}(q) $) and then its image under $f^{-k}$ ($ f^{k}$).
\end{proof}

\section{Main results}\label{sec:main-results}
We now prove that $\cl\ustmani{X}$ is a strange attractor in a certain parameter region. The main tool in the proof is the renormalization model introduced by Dyi-Shing Ou. He essentially proved, that the first return map to $H_{0}$ has a special form, locally sharing many properties with the Lozi family. From now on, unless otherwise stated, we assume \ref{c6}. Let $H=\bigcup_{n=0}^{\infty}f^{n}(H_{0})$.
\subsection*{The renormalization model}
Note that Ou considered family that is conjugate to the orientation preserving Lozi family in the form presented in this article. Nonetheless, all the properties of renormalization model are preserved. Note first, that it follows from \cite{dyi-shing-ou:critical-points-I} that $H:=\bigcup_{j\geq0}\inter f^{j}(H_{0})$ is, in fact, a finite union, and hence a compact set. Let $p\in\nat$ be the smallest satisfying $H=\bigcup_{0\leq j\leq p}\inter f^{j}(H_{0})$. Moreover, let $q\in\nat$ be the smallest satisfying $\cl W^{u}_{X}\subset H_{X}:=\bigcup_{0\leq j\leq q}\inter f^{j}(H_{0})$. Naturally, $q\leq p$.

Let $U_{1}:=f(H_{0})$. For $n>1$, define inductively $U_{n}:=f(U_{n-1})\cap H_{0}$ and $C_{n}:=f^{-n}(C_{n})$. The numbers $p$ and $q$ can be alternatively defined as the smallest numbers with $H_{0}\cap C_{p}\neq \emptyset$ and $\cl\ustmani{X}\cap C_{q}\neq\emptyset$. We will signify the dependence on parameters by $p=p(a,b)$ and $q=q(a,b)$. On the line $L_{\eta}=\real\times\{\eta\}$, for $\eta\in \real$, we consider the ordering inherited from the projection on $\real$. Let $R_{l}, R_{r}\subset \plane$. We will say that $R_{r}$ is to the right (left) of $R_{l}$ if for any $\eta\in\real$, $R_{r}\cap L_{\eta}$ is greater (lesser) then $R_{l}\cap L_{\eta}$. The existence of the renormalization model can be summarised in the following fashion. Originally results were stated for parameters satisfying $a>-3b+1\text{ and  }-1\leq b \leq 0$, which readily holds whenever \ref{c6} is assumed. Majority of statements made in this section are direct consequences of the construction in \cite{dyi-shing-ou:critical-points-I}. Nonetheless, the construction itself requires more notation than is necessary for our purposes. Hence, we do not present it here and refer the reader to \cite{dyi-shing-ou:critical-points-I} for details.
\begin{prop}\label{prop:renormalization-main}
	The first return map $h\colon H_{0}\to H_{0}$ is well defined and the following holds for $n=2,...,p-1$ (see Fig. \ref{fig:Cn-Un})
	\begin{enumerate}[label=(\roman*)]
		\item\label{renorm:prop1} the sets $C_{n}$, are nonempty rectangles with boundary comprising of segments $\fr C_{n}^{u,-}$, $\fr C_{n}^{u,+}$ of $\ustmani{Y}$, positioned respectively in the lower and upper half plane, and segments $\fr C_{n}^{s,l}$, $\fr C_{n}^{s,r}$ of $\stmani{X}$, with $\fr C_{n}^{s,l}$ to the left of $\fr C_{n}^{s,r}$,
		\item\label{renorm:prop2} $C_{p}$ is a triangle with the boundary comprising of segments $\fr C_{p}^{u,-}$, $\fr C_{p}^{u,+}$ of $\ustmani{Y}$, positioned respectively in the lower and upper half plane, and the segment $\fr C_{p}^{s,l}$ of $\stmani{X}$,  which lies to the left of $\fr C_{p}^{u,-}\cup\fr C_{p}^{u,+}$,
		\item\label{renorm:prop3} $C_{n}$ is to the left of $C_{n+1}$, and $C_{n'}\cap C_{m'}=\fr C_{n'}^{s,r}=\fr C_{m'}^{s,l}$ for $n'=m'-1$ and otherwise empty
		\item\label{renorm:prop4} $h_{n}:=h|_{C_{n}}=f^{n}|_{C_{n}}$ and $h\colon C_{n}\to U_{n}$ is a piecewise affine map for $n=2,...,p$
		\item\label{renorm:prop5} there is a segment $R_{n}\subset \inter C_{n}$ with endpoints on $\fr C_{n}^{u,-}$ and $\fr C_{n}^{u,+}$ for which $h_{n}$ is affine on components $C_{n}^{l}$ and $C_{n}^{r}$ of $C_{n}\setminus R_{n}$. Moreover, $h(R_{n}):=T_{n}\subset \{y=0\}$,
		\item\label{renorm:prop6} $C_{n}^{l}$ is to the left of $C_{n}^{r}$,
		\item\label{renorm:prop7} $C_{n}$ is a union of two rectangles bordering on $R_{n}$
		\item\label{renorm:prop8} $U_{n}$ is a union of two rectangles, namely $U_{n}^{-}:=h(C_{n}^{l})$ and $U_{n}^{+}:=h(C_{n}^{r})$ bordering on $T_{n}$, positioned, respectively, in the lower and upper half planes,
		\item\label{renorm:prop9} $\fr U_{n}^{u,r}:=h(\fr C_{n}^{u,-})$ is to the right of $\fr U_{n}^{u,l}:=h(\fr C_{n}^{u,+})$
	\end{enumerate}
\end{prop}

Note that Proposition \ref{prop:renormalization-main} does not say much about the image of $C_{p}$. It turns out that the behaviour of $h_{p}\colon C_{p}\to U_{p}$ determines whether $\Aaa=\cl\ustmani{X}$. While it may be possible, presenting more involved argument, to prove $\Aaa=\cl\ustmani{X}$ under some weaker restrictions of the parameter space, we will work assuming that the vertex $A$ of $H_{0}$ is close to heteroclinic tangency with $\stmani{X}$ and $b<0$ is small enough. Intuitively, the parameter $a$ is responsible for the shift to the right of $\Aaa$, while on the parameter $b$ depends diameter of cross-section $H\cap \{y=0\}$. Hence, parameters for which $A$ is close to tangency form an open set accumulating on $(2,0)$. To formalize the above, let us state the following consequence of Proposition 6.1 from \cite{dyi-shing-ou:critical-points-I}.
\begin{lem}\label{lem:prelim-tangency-param}
	One can find sequences $(a_{n})_{n\geq3}$ and $(\epsilon_{n})_{n\geq3}$, with $\lim_{n\to\infty}a_{n}= 2$ and $\lim_{n\to\infty}\epsilon_{n}=0$, such that if $(a,b)\in B_{n}:=\ball{a_{n},0}{\epsilon_{n}}\cap \Uee_{-}$ then either
	\begin{enumerate}[label=(T\arabic*)]
		\item\label{tangency-c1} $q(a,b)=p(a,b)=n$ and $h(H_{0})\cap\{y=0\}\subset C^{r}_{n}$ (see Fig. \ref{fig:before-tangency1} and \ref{fig:before-tangency2}), or
		\item\label{tangency-c2} $q(a,b)=p(a,b)=n+1$ and $h(H_{0})\cap\{y=0\}\subset C^{l}_{n+1}$ (see Fig. \ref{fig:after-tangency1} and \ref{fig:after-tangency2}).
	\end{enumerate}
	Moreover, we have $h(H_{0}\cap\{y=0\})\subset U_{2}$.
\end{lem}
\begin{proof}
	For the case of the degenerate Lozi family, that is the case of $b=0$, restricting the parameter set to $\{(a,b)\colon b=0\}$ one obtains the lemma as a direct consequence of Prop. 6.1 from \cite{dyi-shing-ou:critical-points-I}. As the point $A$ and intersections of $\stmani{X}$ with the $x$-axis depend continuously on parameters, we can extend our argument by continuity, proving the lemma. 
\end{proof}
We can now list crucial properties that will determine the behaviour of $h_{p}$.
\begin{lem}\label{lem:behaviour-hp}
	There exists an open set of parameters $\Uee_{\text{tan}}$ with $(2,0)$ in the boundary such that either $U_{p}\subset C_{2}$ (see Fig. \ref{fig:after-tangency1} and \ref{fig:after-tangency2}) or the following holds (see Fig. \ref{fig:before-tangency1} and \ref{fig:before-tangency2})
	\begin{enumerate}[label=(P\arabic*)]
		\item\label{Cp->Up(1)} there is a segment $R_{p}\subset \inter C_{p}$ with endpoints on $\fr C_{p}^{u,-}$ and $\fr C_{p}^{u,+}$ for which $h_{p}$ is affine on components $C_{p}^{l}$ and $C_{p}^{r}$ of $C_{p}\setminus R_{p}$. Moreover, $h(R_{p}):=T_{p}\subset \{y=0\}$,
		\item\label{Cp->Up(2)} $U_{p}$ is a broken triangle, positioned to the right of $U_{n}$, for $n=2,...,p-1$, and comprising of a rectangle $U_{p}^{-}:=h(C_{p}^{l})$ and a triangle $U_{p}^{+}:=h(C_{p}^{r})$, positioned respectively in the lower and upper half planes,
		\item\label{Cp->Up(3)}  the boundary $\fr U_{p}$ comprises of the segment $h(\fr C^{s,l}_{p})\subset \stmani{X}$ and broken segments $\fr U_{p}^{u,l}=h(\fr C_{p}^{u,+})$ and $\fr U_{p}^{u,r}=h(\fr C_{p}^{u,-})$ with the former lying to the left of the latter,
		\item\label{Cp->Up(4)} the triangle $U_{p}^{+}$ connects $C_{p}$ with $C_{2}$ and $h(A)\in C_{2}$.
	\end{enumerate}
\end{lem}
\begin{proof}
	We first restrict the parameter set to the one satisfying \ref{c6}. Applying Lemma \ref{lem:prelim-tangency-param} we obtain sequences $(a_{n})_{n>1}$ and $(\epsilon_{n})_{n>1}$ and corresponding sequence of open balls $B_{n}:= \ball{a_{n},0}{\epsilon_{n}}$. We claim that $\Uee_{\text{tan}}=\bigcup_{n>1}B_{n}$ is the required set of parameters. 
	
	Let us first note that \ref{Cp->Up(1)} follows simply from \cite{dyi-shing-ou:critical-points-I}.
	
	Fix $n>1$ and let $(a,b)\in B_{n}$. We then have $p(a,b)=q(a,b)$. If the case \ref{tangency-c2} occurs, then, using property \ref{renorm:prop2} from Prop. \ref{prop:renormalization-main}, naturally $h(A)$ will be close to $h(\fr C_{p}^{s,l})\subset \seg{ED}$ and so $C_{p}\subset U_{2}$.
	
	Assume now that the case \ref{tangency-c1} occurs. By property \ref{Cp->Up(1)} $U_{p}$ must comprise of a triangle and a rectangle, adjacent to each other along $T_{p}$. The rectangle must be positioned in the lower half plane, since it follows from \cite{dyi-shing-ou:critical-points-I} that $\fr U_{p}^{s,l}=\stmani{X}\cap\fr U_{p}$ lies lower on $\seg{ED}$ then the fixed point $X$. Considering any orientation of $\fr C_{p}$ and taking into account fact that $h$ preserves orientation, the only possibility is that $\fr U_{p}^{u,l}$ is to the left of $\fr U_{p}^{u,r}$.
	
	The last property follows from $h(H_{0}\cap\{y=0\})\subset U_{2}$ assured by Lemma \ref{lem:prelim-tangency-param}.
\end{proof}
\begin{figure}[p]
	\caption{Renormalization in the case \ref{tangency-c2} where $U_{p}\subset C_{2}$}
	\centering
	\includegraphics[width=\textwidth]{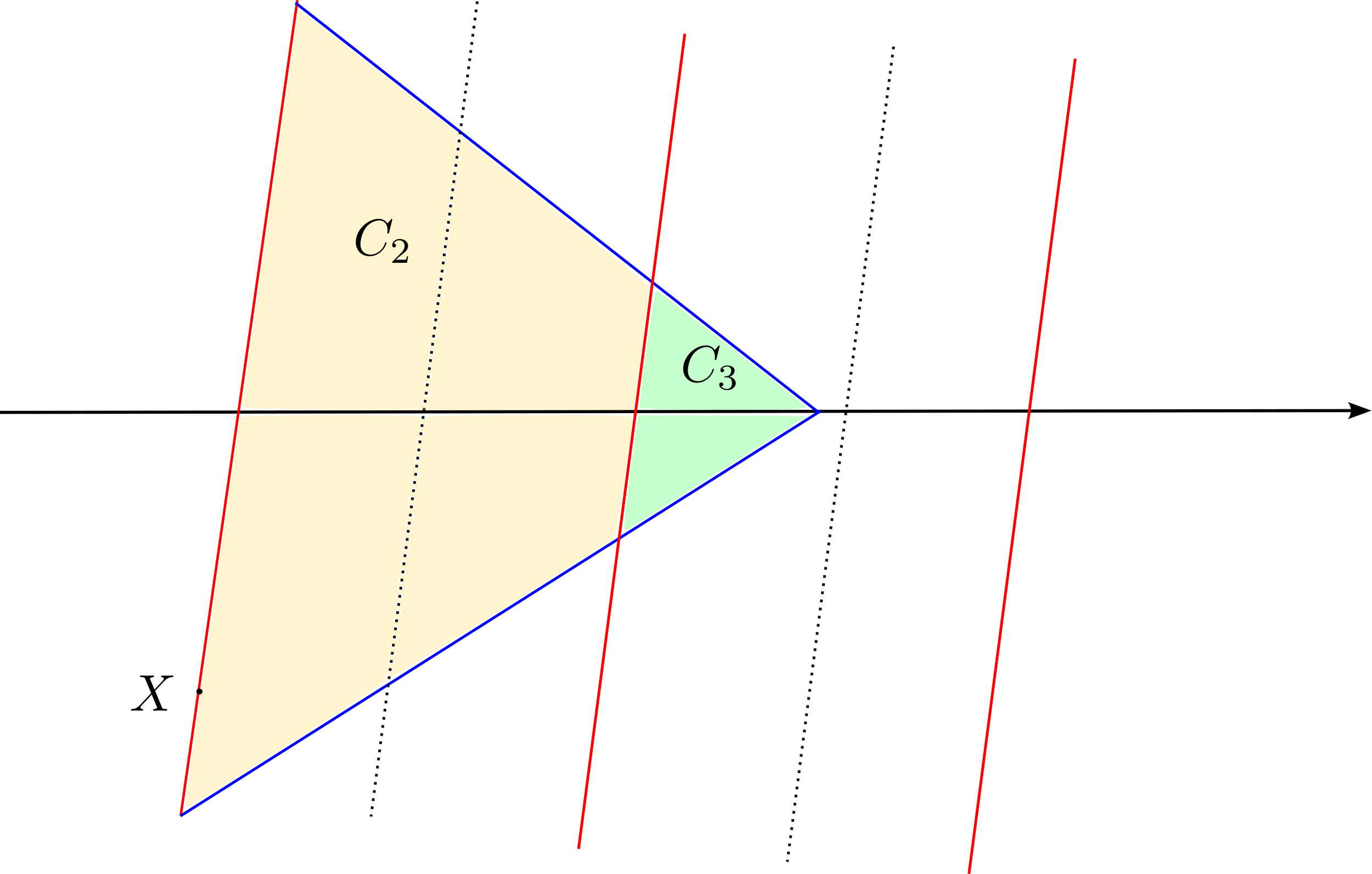}
	\label{fig:after-tangency1}
\end{figure}
\begin{figure}[p]
	\caption{The image of $H_{0}$ under $h$ for the case \ref{tangency-c2}}
	\centering
	\includegraphics[width=\textwidth]{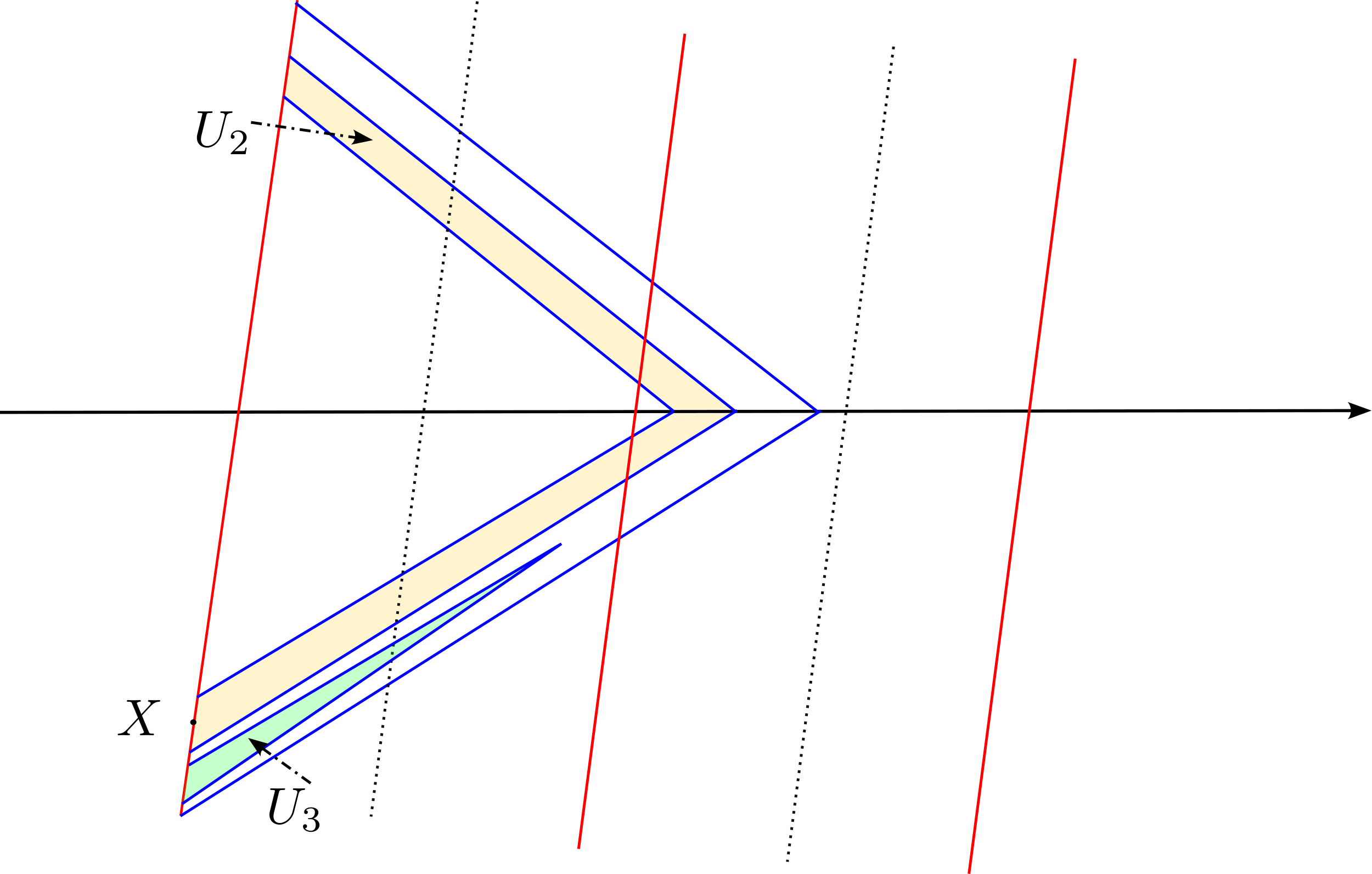}
	\label{fig:after-tangency2}
\end{figure}
\begin{figure}[p]
	\caption{Renormalization in the case alternative to $U_{p}\subset C_{2}$, \ref{tangency-c1}}
	\centering
	\includegraphics[width=\textwidth]{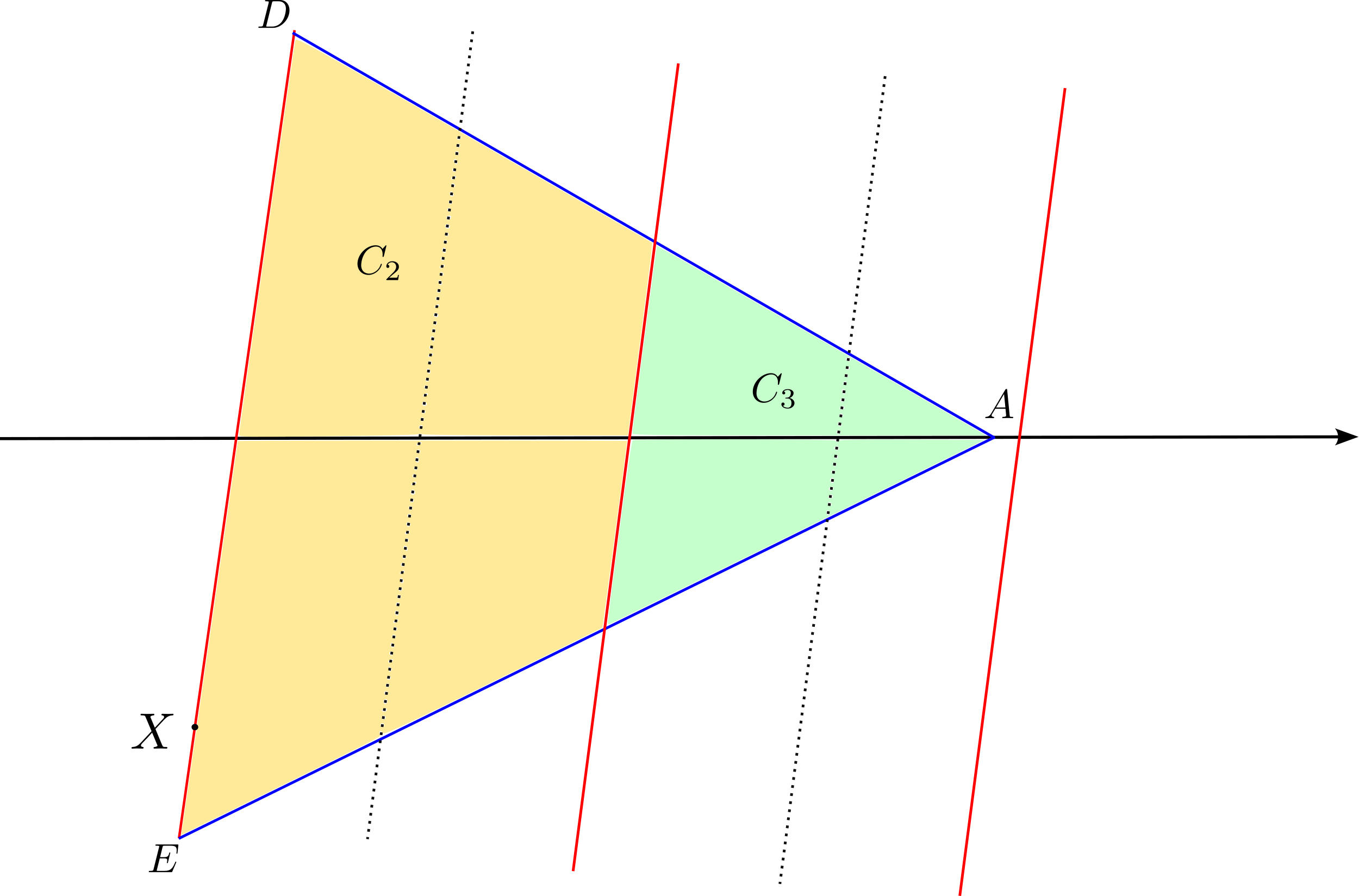}
	\label{fig:before-tangency1}
\end{figure}
\begin{figure}[p]
	\caption{The image of $H_{0}$ under $h$, the case \ref{tangency-c1}. The purple dashed line signifies the region $W'$}
	\centering
	\includegraphics[width=\textwidth]{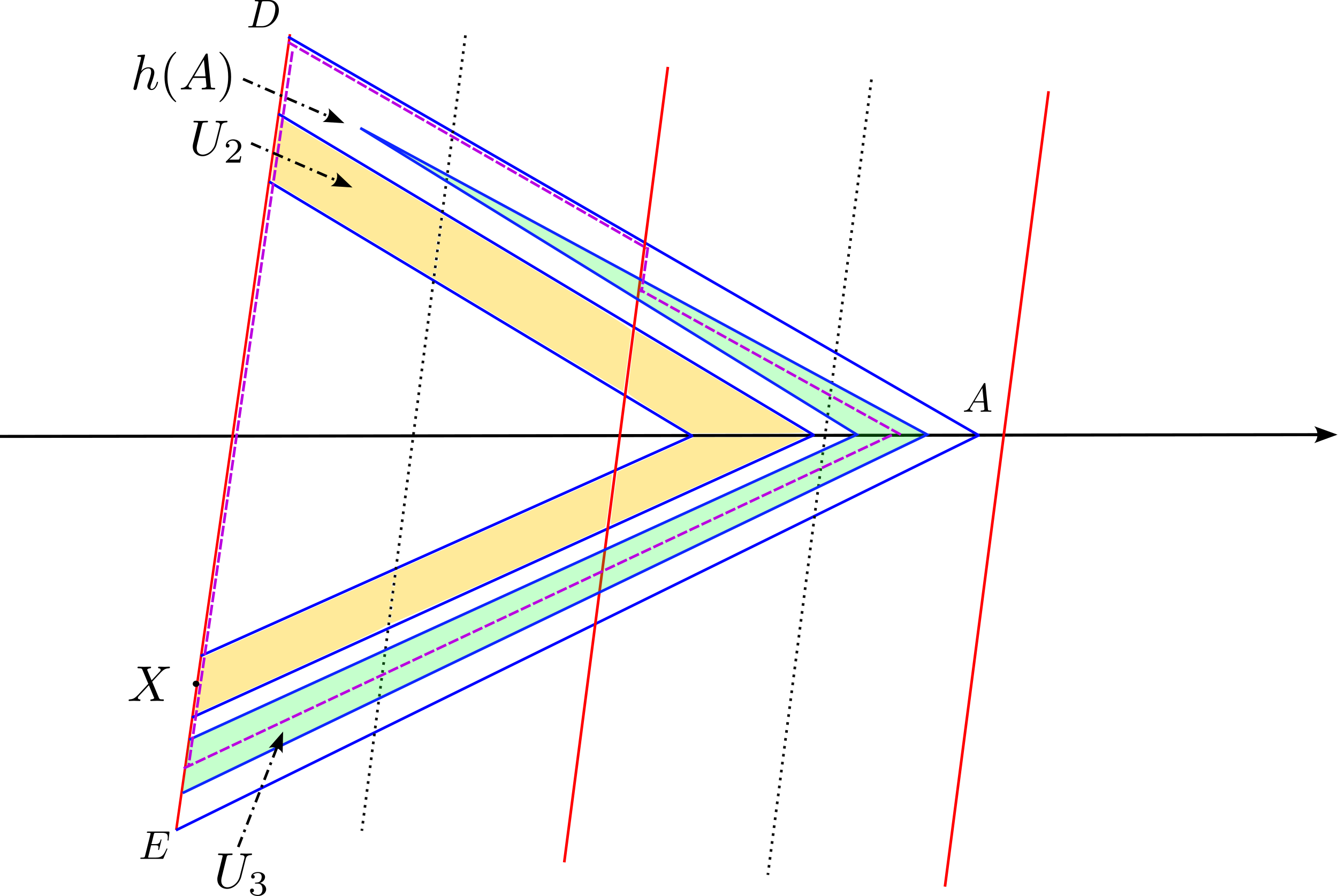}
	\label{fig:before-tangency2}
\end{figure}
\begin{figure}[p]
	\caption{The set $C_{2}$ together with its image $U_{2}$}
	\centering
	\includegraphics[width=\textwidth]{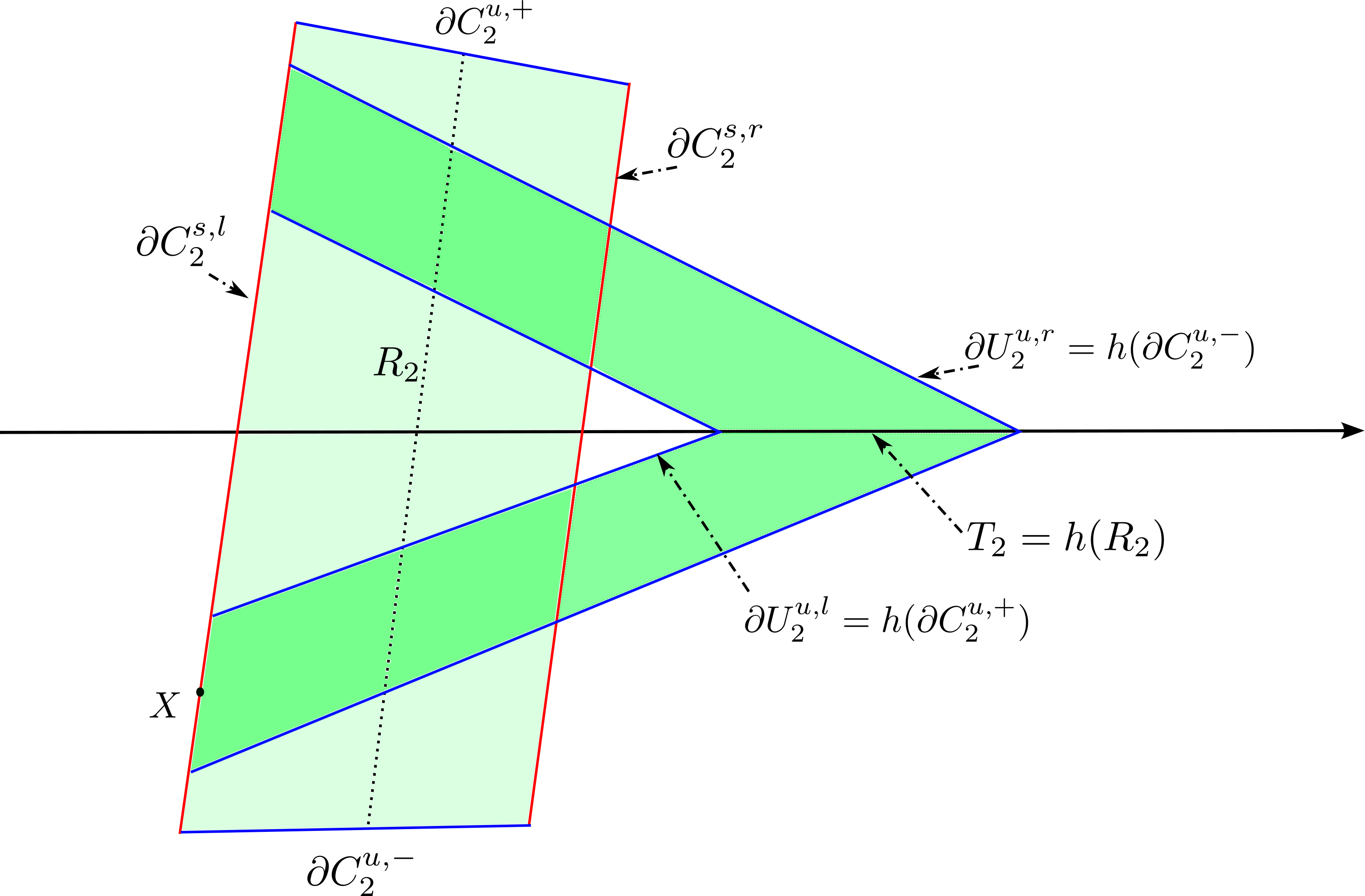}
	\label{fig:Cn-Un}
\end{figure}
\begin{prop}
	For parameters in $\Uee_{\text{tan}}$ we have $\bigcap_{i\in\nat}h^{i}(H_{0})=H_{0}\cap\cl \ustmani{X}$.
\end{prop}
\begin{proof}
	Note that to prove the proposition it is enough to show $\Raa (\bar{i},\bar{t}):=\bigcap_{k\in\nat}h^{i_{k}}(C_{t_{k}})\subset \cl \ustmani{X}$, where $\bar{i}=\{i_{k}\}_{k\in\nat}$ is any increasing sequence with $i_{0}=0$ and $\bar{t}=\{t_{k}\}_{k\in\nat}$ satisfies $t_{k}\in\{1,..,p\}$. For later use, we define $\Raa (\bar{i},\bar{t},n):=\bigcap_{k\leq n}h^{i_{k}}(C_{t_{k}})$ for $n\in\nat$.
	
	Moreover, it is visible that if $h^{j}(C_{t'})\subset W$ for some disc $W$ with $\omega(W)\subset \cl\ustmani{X}$, $j\in\nat$ and $t'\in\{1,...p\}$, and if $t'=t_{k}$ for infinitely many $k\in\nat$, we must have $\Raa(\bar{i},\bar{t})\subset \cl\ustmani{X}$. At last, note that any disc $W$ with $\fr W\subset \ustmani{X}\cup\stmani{X}$, must also have $\omega(W)\subset\cl\ustmani{X}$.
	
	Note that by Lemma \ref{lem:prelim-tangency-param} one can find a curve $\gamma\subset C_{p-1}\cap \ustmani{X}$ which connects $\fr C^{u,l}_{p-1}$ and $\fr C^{u,r}_{p-1}$. Thus, $h(\gamma)$ connects $h(\fr C^{u,l}_{p-1})$ and $h(\fr C^{u,r}_{p-1})$, and by \ref{renorm:prop6} $h(\gamma)$ is to the right of every $U_{j}$, for $j=2,...,p-1$. Hence, $h(R)\subset W$, where $W$ is bounded by $\seg{ED}\subset\stmani{X}$ and $h(\gamma)\subset \ustmani{X}$, and $R=H_{0}\setminus (C_{p-1}\cup C_{p})$. It proves $\Raa(\bar{i},\bar{t})\subset\cl\ustmani{X}$ in the case $t_{k}<p-1$ for $k\in\nat$. 
	
	We split our discussion into two cases, according to conditions of Lemma \ref{lem:behaviour-hp}. Assume first, that $U_{p}\subset C_{2}$. Then, by our previous discussion, since $C_{2}\subset R$, there must be $\Raa(\bar{i},\bar{t})\subset \cl\ustmani{X}$ if $t_{k}=p$ for infinitely many $k\in\nat$. Consider now sets $\Raa(\bar{i},\bar{t})$ with $t_{k}= p-1$ for infinitely many $k\in\nat$. It is enough to treat sets with $t_{k}= p-1$ for every $k\in\nat$. We will show that any sequence $\bar{i}$ defines $\Raa(\bar{i},\bar{t},n)$, $n\geq 0$, which is a disjoint union of rectangles $T$ and every $T$ has its opposite sides connected by $\ustmani{X}$. As $h$ is area contracting, it shows $\Raa(\bar{i},\bar{t})\subset\cl\ustmani{X}$. Using the last claim of Lemma \ref{lem:prelim-tangency-param} we see that $\Raa(\bar{i},\bar{t},1)$ is a disjoint union of rectangles $T$ with mentioned form. Note that the image $h(T)$ of a rectangle $T$ is a union of two rectangles bordering on $T_{p-1}$, both connecting $h(\fr C_{p-1}^{s,l})$ with $h(\fr C_{p-1}^{s,r})$. Moreover, $h(T)\cap  T_{p-1}\subset C_{p}$. Consequently, $h(T)\cap C_{p-1}$ is a union of two rectangles with prescribed properties. Proceeding by induction, we see that one can apply the above reasoning for any $n\in\nat$, proving our claim.
	
	We turn to the case, when properties \ref{Cp->Up(1)} to \ref{Cp->Up(4)} hold. Using Lemma \ref{lem:prelim-tangency-param} and property \ref{Cp->Up(4)} one can find $\tau\subset C_{p}\cap\ustmani{X}$ with $h(\tau)$ connecting $\seg{ED}$, $T_{p}$ and $\fr C_{2}^{s,r}$. Let $g\in h(\tau)\cap \fr C_{p-1}^{s,r}$. Let $S_{g}\subset \fr C^{s,r}_{p-1}$ be the segment connecting $g$ and $\fr C^{u,+}_{p-1}$. Consider the region $W'$ (see Fig. \ref{fig:before-tangency2}) bounded by $\seg{ED}$, $h(\tau)$, $S_{g}$ and $\bigcup_{j=2}^{p-1}\fr C^{u,+}_{j}$. Note that $\omega(\fr W')\subset \cl\ustmani{X}$. It follows from \ref{Cp->Up(2)} and \ref{renorm:prop9} that $h(C_{j})\subset W'$ for every $j=2,...,p-1$. Hence, $\Raa(\bar{i},\bar{t})\subset\cl\ustmani{X}$ if only $t_{k}<p$ for infinitely many $k\in\nat$. Therefore, it remains to consider the case of $t_{k}=p$ for every $k\in\nat$. As before, it is not hard to prove, using induction, Lemma \ref{lem:behaviour-hp} and \ref{lem:prelim-tangency-param} that, for every $n\in\nat$, $\Raa(\bar{i},\bar{t}, n)$ is a union of rectangles $T$ which have opposite sides connected by a curve in $\ustmani{X}$. Consequently, as $h$ contracts the area, $\Raa(\bar{i},\bar{t})\subset\cl\ustmani{X}$, proving the proposition.
\end{proof}

We have an immediate
\begin{cor}\label{trapping-region}
	For parameters in $\Uee_{\text{tan}}$ we have $\bigcap_{n=0}^{\infty}f^{n}(H)=\cl\ustmani{X}$.
\end{cor}
As a consequence we obtain

\begin{thm}
	For parameters in $\Uee_{\text{tan}}$ the closure of the unstable manifold $\Aaa=\cl \ustmani{X}$ of $X$ is a topological attractor.
\end{thm}
\begin{proof}
	We will find an open set $V\supset \Aaa$ with $f(\cl V)\subset V$ and $V\subset H$. It will conclude our proof, since by Corollary \ref{trapping-region} $\bigcap_{n=0}^{\infty}f^{n}(H)=\Aaa$.
	
	As $\Aaa\subset \inter H$, we have $f^{p+1}(H)\subset \inter H$. Hence we can find an open set $V_{0}$ with $\cl V_{0}\subset \inter H$ and $f^{p+1}(H)\subset V_{0}$. We inductively define, for $0<n\leq p$, open sets $V_{n}$ with $\cl V_{n}\subset f^{n}(\inter H)$ and $f(\cl V_{n-1})\subset V_{n}$. Assume we have constructed $V_{n-1}$. Note that $f(\cl V_{n-1})\subset f^{n}(\inter H)$. Therefore, it is enough to take sufficiently small neighbourhood $V_{n}\subset f^{n}(\inter H)$ of $f(\cl V_{n-1})$ so that $\cl V_{n}\subset f^{n}(\inter H)$. Let $V=\bigcup_{n=0}^{p}V_{n}$. We compute \[
	f(\cl V)\subset \bigcup_{n=0}^{p}f(\cl V_{n})\subset \bigcup_{n=1}^{p}V_{n}\cup f^{p+1}(H)\subset V.
	\] Notice that $V$ may not be an open disc. Therefore to construct attracting neighbourhood of $\Aaa$ fitting definition from Section \ref{sec:attractors} we have to add to $V$ all bounded components of its complement.
\end{proof}
\begin{lem}\label{4}
	Let $I\subset G$ be a segment contained in some unstable manifold. Then there exist $n\geq 0$ and a segment $I_{1}\subset f^n(I)$ such that $I_{1}$ intersects both coordinate axis.
\end{lem}
\begin{proof}
	Misiurewicz's proof from \cite{strange-attractor-mis} follows almost without any changes. In the case of $b<0$ we have different constant for expanding vectors in the unstable manifold, hence in the place of assumption $b/c>\sqrt2$ present in Misiurewicz work, we have to put $\beta>\sqrt2$, that is \ref{c5}. Condition \ref{c5} together with \ref{c2} and \ref{c3} are equivalent to \ref{c6}.
		
	Suppose that such $n$ and $I_{1}$ do not exist. The set $f^k(I)$ is a broken line. The mapping $ f $ is linear in the left and right half-planes. If a segment $J\subset f^k(I)$ intersects the vertical axis, then $ f(J)$ is a union of at most two segments, each of them intersecting the horizontal axis. Then $f(J)$ does not intersect the vertical axis. Consequently, $f^2(J)$ consists of, at most, two segments. If $J$ does not intersect the vertical axis, then $f(J)$ is a segment and $ f^2(J) $ consists of, at most, two segments. Hence, in both cases, $ f^2(J) $ consists of, at most, two segments. Thus, there are no more than $2^k$ segments in $ f^{2k}(J) $ for $k=0,1,2,....$ Since $I\subset G$, $f^{2k}(I)\subset G$ as well. Therefore, the length of $f^{2k}(I)$ is at most $2^k \diam (G)$. On the other hand, $I$ is a segment of an unstable manifold, thus, the length of $f^{2k}(I)$ is at least $\beta^{2k}\diam I$. If $k$ is large enough, we get $\beta^2\leq 2$, a contradiction with \ref{c5}.
\end{proof}

\begin{thm}
	The Lozi map is topologically mixing on the unstable manifold of $X$.
\end{thm}
\begin{proof}
		Let $U$, $V$ be open sets in $G$, intersecting $\ustmani{X}$ non empty.
	\begin{claim*}
		There exists $k_{1}\in \nat$ such that $f^{-k}(V)$, for $k>k_{1}$, intersects $f(G)$ in such a way that every segment in $f(G)$ intersecting both coordinate axes also intersects $f^{-k}(V)$.
	\end{claim*}
	\begin{proof}
		Note that every segment in $f(G)$ intersecting both coordinate axes, must have some points $S^{x}\in\{y=0\}$ and $S^{y}\in\{x=0\}$, for which its subsegment $\seg{S^{x}S^{y}}$ is entirely contained either in $\{y\geq 0\}\cap f(G)$ or in $\{y\leq 0\}\cap f(G)$. Moreover, since $f(M)\in H_{0}$, there is also $S^{x}\in H_{0}$. Hence $\seg{S^{x}S^{y}}$ must intersect $\seg{ED}\subset \stmani{X}$. Similarly, $\seg{S^{x}S^{y}}$ will intersect any arc contained in some small neighbourhood of $\seg{ED}$, that connects neighbourhood of $E$ with that of $D$. Thus, it is enough to prove that $f^{-k}(V)$ contains such an arc for $k>k_{1}$ and $k_{1}\in \nat$ big enough.
		
		Note that any small rectangle $R$ adjacent to $X$, which boundary comprises of segments of $\ustmani{X}$ and $\stmani{X}$, under iterations will be stretched along $\seg{ED}$. As $\ustmani{X}\cap V\neq\emptyset$, we must also have $R\cap f^{-k}(V)\neq\emptyset$ for $k>k_{1}$ and $k_{1}\in\nat$ big enough. Consequently, $f^{-k}(V)$ will accumulate on $\seg{ED}$, proving our claim.
	\end{proof}
	Since $U\cap W^{u}_{X}$ is nonempty, there exists a segment $I\subset U\cap W^{u}_{X}$. By Lemma \ref{4}, there exists $n_{1}\geq 0$ such that some subsegment of $f^{n_{1}}(I)$ intersects both coordinate axes. Therefore, it also intersects $f^{-k}(V)$ for all $k\geq k_{1}$. Hence, $f^n(U)\cap V \cap W_{X}^{u}\supset f^{n-n_{1}}(f^{n_{1}}(I)\cap f^{n_{1}-n}(V))\neq\emptyset$ for all $n\geq n_{1}+k_{1}$.
\end{proof}
The following theorem for the family of 2-dimensional border-collision normal forms was obtained by Ghosh and Simpson \cite{simpson:devaney-chaos}.
\begin{thm}\label{WsX-dense-in-G}
	The stable manifold $W^{s}_{X}$ is dense in $G$.
\end{thm}
\begin{proof}
	Note that segment $\overline{ED}$ has property that every segment in $f(G)$ intersecting both coordinate axis intersects $\overline{ED}$. Now, to prove that $W^{s}_{X}\cap G$ is dense in $G$ it is enough to use the aforementioned property, Lemma \ref{4} and Theorem \ref{hyperbolicity}.
\end{proof}

\section{Continuity of Lozi attractors in the Hausdorff metric}\label{sec:continuity}
Recall that the Hausdorff distance of $H,L\subset \mathbb R^2$ is defined by
\[ \dH{H}{L}=\max\{d_{a}(H,L),d_{a}(L,H)\}, \]where $d_{a}(H,L)=\sup_{h\in H}\inf_{l\in L}d(h,l)$.

Let $f_{\mu}$ be a family of maps on $\mathbb R^2$ that vary continuously with respect to the parameter $\mu\in M\subset \mathbb{R}^2$, where $M$ is compact with nonempty interior. We assume that for every $\mu\in M$ there is an attractor $\Aaa_{\mu}\subset \mathbb R^2$ and the following hold.
\begin{enumerate}[label=(A\arabic*)]
	\item\label{A1} There exists a compact $\Omega\subset \mathbb R^2$ such that $\Aaa_{\mu}\subseteq \Omega$ for every $\mu\in M$.
	\item\label{A2} For every $\mu\in M$ there exists a compact $H_{\mu}\subseteq \Omega$, continuous in $\mu$ with respect to the Hausdorff metric, such that $f_{\mu}(H_{\mu})\subseteq H_{\mu}$ and $\Aaa_{\mu}=\bigcap_{n=0}^{\infty}f^{n}_{\mu}(H_{\mu})$
\end{enumerate}
Key lemma, that we are going to use, assures us that under conditions \ref{A1} and \ref{A2} attractors vary continuously, if we additionally have uniform convergence of iterations of attractors' neighbourhoods.

\begin{lem}[{\cite[Thm. 5.1]{glendinning2019robust}}]\label{glen-key}
	If the conditions \ref{A1} and \ref{A2} hold and $\dH{f^{n}_{\mu}(H_{\mu})}{\Aaa_{\mu}}\to 0$ as $n\to\infty$ uniformly in $M$, then $\Aaa_{\mu}$ is continuous in the Hausdorff metric in $M$.
\end{lem}

Let us now confine our attention to the family of Lozi maps. From now on we define $f_{\mu}$ to be the Lozi map with parameter $\mu=(a,b)\in \Uee_{\text{tan}}$. Denote by $\Aaa_{\mu}$ the attractor for $f_{\mu}$. We have obvious candidates for attracting neighbourhoods from the condition \ref{A2}, that is polygons $H$, which we denote by $H_{\mu}$. It is clear that $H_{\mu}$ varies continuously in $\Uee_{\text{tan}}$ with respect to the Hausdorff metric.
\begin{lem}
	The attractor $\Aaa_{\mu}$ for the Lozi map is continuous in the Hausdorff metric.
\end{lem}
\begin{proof}
	Let $\mu_{0}\in \Uee_{\text{tan}}$. By Lemma \ref{glen-key} it suffices to prove the uniform convergence of $f_{\mu}^{n}(H_{\mu})$ in some neighbourhood of $\mu_{0}$, for $\mu=(a,b)$. We follow closely the proof of Theorem 7.2 from \cite{glendinning2019robust}. Let $K_{\mu}=\lebmes{H_{\mu}}$. Then, since the Jacobian of $f_{\mu}$ is equal to $b$, $\lebmes{f_{\mu}^{n}(H_{\mu})} \leq |b|^{n}K_{\mu}$. 
	
	Let $x\in f^{n}_{\mu}(H_{\mu})$. As every ball $\ball{x}{r}$, $r>0$, contained in $f^{n}_{\mu}(H_{\mu})$ must satisfy $\lebmes{\ball{x}{r}}\leq \lebmes{f^{n}_{\mu}(H_{\mu})}$, we obtain a bound on radius $r$, that is $r\leq\sqrt{K_{\mu}/ \pi}|b|^{n/2}$, which in turn reveals
	\begin{equation}\label{one}
	 d(x,\fr f^{n}_{\mu}(H_{\mu}))\leq \sup\{r>0\colon \ball{x}{r}\subset f^{n}_{\mu}(H_{\mu})\}\leq\sqrt{\dfrac{K_{\mu}}{\pi}}|b|^{n/2} . 
	\end{equation}Since $\Aaa_{\mu}\subset H_{\mu}$, \eqref{one} can be applied to points of $\Aaa_{\mu}$, giving
	\[ d_{H}(\Aaa_{\mu},\fr f^{n}_{\mu}(H_{\mu}))\leq \sqrt{\dfrac{K_{\mu}}{\pi}}|b|^{n/2} .\] Consequently, for any $x\in f^{n}_{\mu}(H_{\mu})$, \[ d_{H}(\Aaa_{\mu},x)\leq d_{H}(\Aaa_{\mu},\fr f^{n}_{\mu}(H_{\mu}))+ d_{H}(x,\fr f^{n}_{\mu}(H_{\mu}))\leq 2\sqrt{\dfrac{K_{\mu}}{\pi}}|b|^{n/2}.\] Thus, \[ d_{H}(\Aaa_{\mu},f^{n}_{\mu}(H_{\mu}))\leq 2\sqrt{\dfrac{K_{\mu}}{\pi}}|b|^{n/2}.\]
	 
	 Let $U\subset \Uee_{\text{tan}}$ be a closed neighbourhood of $\mu_{0}$. Denote $K_{\max}=\sup_{\mu\in U}K_{\mu}$ and $b_{\max}=\sup\{|b|\colon\mu=(a,b)\in U\}$. Then
	\[ d_{H}(\Aaa_{\mu},f^{n}_{\mu}(H_{\mu}))\leq 2\sqrt{\dfrac{K_{\max}}{\pi}}{b_{\max}}^{n/2}\] for every $\mu\in U$. As $b_{\max}<1$ we have proved uniform convergence around $\mu_{0}$ and the claim.
\end{proof}

\begin{cor}
	The prime ends rotation number $\exrotation{f_{\mu},\Aaa_{\mu}}$ of the family of Lozi attractors $\{\Aaa_{\mu}\}_{\mu\in \Uee_{\text{tan}}}$ vary continuously with $\mu$.
\end{cor}
\begin{proof}
	This is a consequence of Proposition 2.2. in \cite{barge-cont}. Although the result in \cite{barge-cont} is for one parameter family, it is clear from the proof that it works for two parameters as well.
\end{proof}
\textbf{Acknowledgements}
This work was supported by the National Science Centre, Poland (NCN), ~grant no. 2019/34/E/ST1/00237. We would like to thank Jan P. Boroński for his constant support, many invaluable comments and fruitful discussions. We would also like to thank Dyi-Shing Ou and Sonja \v{S}timac for their suggestions and proofreading first versions of manuscript. Lastly, we thank David Simpson for bringing our attention to recent relevant results on the border collision normal form family.
\bibliography{lozi_map_bib}
\bibliographystyle{alphaurl}

\end{document}